\documentclass[11pt, draft]{amsart}
\usepackage{amssymb, amstext, amscd, amsmath, paralist}
\usepackage{xypic, color}
\usepackage{graphicx}
\usepackage{tikz}
\usetikzlibrary{matrix}

\makeatletter
\def\@cite#1#2{{\m@th\upshape\bfseries%
[{#1\if@tempswa{\m@th\upshape\mdseries, #2}\fi}]}}
\makeatother

\theoremstyle{plain}
\newtheorem{theorem}{Theorem}[section]
\newtheorem{corollary}[theorem]{Corollary}
\newtheorem{proposition}[theorem]{Proposition}
\newtheorem{lemma}[theorem]{Lemma}

\theoremstyle{definition}
\newtheorem{definition}[theorem]{Definition}
\newtheorem{example}[theorem]{Example}

\newtheorem{remark}[theorem]{Remark}

\theoremstyle{remark}

\def\bbC{\mathbb C}
\def\bbD{\mathbb D}

\def\bbN{\mathbb N}

\def\bbR{\mathbb R}
\def\bbT{\mathbb T}
\def\bbZ{\mathbb Z}

     \newcommand{\sA}{\mathcal A}
     \newcommand{\sB}{\mathcal B}
     \newcommand{\sC}{\mathcal C}

     \newcommand{\sH}{\mathcal H}

\newcommand{\al}{\alpha}
\newcommand{\be}{\beta}
\newcommand{\ep}{\varepsilon}
\newcommand{\vpi}{\varphi}
\newcommand{\si}{\sigma}

\newcommand{\ga}{\gamma}

\newcommand{\Aut}{\operatorname{Aut}}

\addtocontents{toc}{\protect\setcounter{tocdepth}{1}}

\begin{document}

\title[Uniform algebras on $\bbT^2$]{Isomorphism of uniform algebras on the 2-torus}

\author[J.R. Peters]{Justin R. Peters}
\address{Department of Mathematics\\Iowa State University\\ Ames\\ Iowa\\ IA 50011\\ USA}
\email{peters@iastate.edu}

\author[P. Sanyatit]{Preechaya Sanyatit}
\address{Department of Mathematics\\Silpakorn University\\ Nakhon Pathom\\ 73000\\ Thailand}
\email{sanyatit$\textunderscore$p@silpakorn.edu}

\subjclass[2010]{Primary 46J10; Secondary 11Z05, 42B35}

\keywords{uniform algebra, Gelfand space, Pell's equation}

\begin{abstract} For $\al$ a positive irrational, let $\sA_{\al}$ be the subalgebra of continuous functions on the two-torus whose Fourier transform vanishes at $(m, n)$ if $m + \al n < 0.$ These algebras
were studied by Wermer and others, who proved properties such as maximality and characterized the Gelfand space. One of the major themes of current work in operator algebras is classification,
but none of the properties which were investigated earlier distinguished  between $\sA_{\al}$ and $\sA_{\be},$ if $\be$ is another positive irrational.  We address this question.  We also determine the
automorphism group of $\sA_{\al}.$
\end{abstract}

\maketitle


\section{Introduction}

Recently there has been remarkable progress in the classification of unital simple C$^*$-algebras:
it is now known that the class of unital
simple separable C*-algebras with finite nuclear dimension
(and satisfying the possibly redundant Universal Coefficient
Theorem, or UCT) is  classified by the
natural K-theoretical invariant (also known
as the Elliott invariant). This is described \cite{EllNiu16}, \cite{EllGoNiu16} and \cite{TiWhWi16}.

For the larger category of operator algebras (which need not be closed under the $^*$ operation), the situation is completely different: the classification results that have been obtained are restricted to narrow classes.
The early work on C$^*$-algebras showed that commutative C$^*$-algebras are classified by their Gelfand spectra.
In the category of commutative operator algebras, there is no classification, nor is such a program
underway.

Of course, unlike commutative C$^*$-algebras, which are semisimple, commutative operator algebras can be radical Banach algebras, so the classification task is formidable,
and indeed the unresolved status of the invariant subspace problem serves as a reminder of how little is understood of even the singly-generated commutative operator algebras.

However, if we focus on the class of commutative operator algebras which are semisimple, and hence for which the Gelfand map is an isometric isomorphism, one might hope that a classification
is within reach: indeed, not only are these algebras commutative, but so are their C$^*$-envelopes.  The algebras in this class are also known as uniform algebras.

In this work we have focused on a class of uniform algebras which are subalgebras of the continuous functions on the
 torus $\bbT^2 $  \ ($\bbT$ the unit circle). Let $\al $ be a positive irrational, and let
$ \sA_{\al} $ be the subalgebra of $C(\bbT^2)$ consisting of functions $f$ such that the Fourier transform
$\hat{f}(m, n) = 0$ if $m + n\al < 0.$  These algebras along with other uniform algebras
were studied in the 1950's and 60's. (See \cite{Gam69}, \cite{Wer61}, \cite{Sto71}, and the references contained therein.) The kinds of questions these authors considered were
harmonic analytic in nature; it is evident that the question of whether two such algebras were isometrically isomorphic was not something they addressed.  However, none of the properties
they studied differentiated between $\sA_{\al}$ and $\sA_{\be}$ if $\al,\ \be$ are two positive irrationals.

Unlike commutative, unital C$^*$-algebras, which are determined by their Gelfand spaces, in the case of uniform algebras there is no analogue of the Banach-Stone theorem.
Indeed, it is easy to show that the Gelfand spaces of all the $\sA_{\al}$ are homeomorphic, yet it is not true that the the $\sA_{\al}$ are all in the same
isometric isomorphism class.
In Theorem~\ref{t:iso} we give the form of an isometric isomorphism of $\sA_{\al} \to \sA_{\be}.$  It follows that the cardinality of the isomorphism class is countable.
Corollary~\ref{c:orderedgroup} shows that there is a group invariant:  let $G_{\al}$ be the dense subgroup of $\bbR$ consisting of $\{m + n\al:\ m, n \in \bbZ \}.$ We show that
$\sA_{\al}$ and $\sA_{\be}$ are isometrically isomorphic if and only if there is a group isomorphism $G_{\al} \to G_{\be}$ which maps $G_{\al}\cap \bbR^+ \to G_{\be}\cap \bbR^+.$

In section 4 we examine the group of isometric automorphisms of $\sA_{\al}.$  If $\al$ is not a quadratic irrational, then $\Aut(\sA_{\al}) \cong \bbT^2.$  However if $\al$ is
a quadratic irrational, then $\Aut(\sA_{\al}) $ is isomorphic to a semidirect product of $\bbT^2$ with $\bbZ$ (Theorem~\ref{t:semidirect}).

In section 5, we give an explicit form for the automorphism group of $\sA_{\al}, $ where $\al$ is a quadratic irrational.  As noted above, $\Aut(\sA_{\al})$ is a semidirect product, where the action
of $\bbZ$ on $\bbT^2$ is given by a matrix in $GL(2, \bbZ).$  The results of this section give the form of the generator in $GL(2, \bbZ)$ of this action in terms of the
quadratic irrational $\al$ and the fundamental solutions of Pell's equations.


\section{Background and Notation}

Let $\bbT^2$ denote the 2-torus $\bbT \times \bbT,$ where $\bbT$ denotes the unit circle. Let $d\mu$ be normalized Lebesgue measure on $\bbT^2.$  If $f: \bbT \to \bbC$ is in $L^2,$
the Fourier transform is a function on $\bbZ^2$ given by

\[ \hat{f}(m, n) = \int_{\bbT^2} f(e^{ i s}, e^{i t}) e^{-i(ms + nt)}\, d\mu \]

We will employ Cesaro means of functions $f \in C(\bbT^2).$  As  there are various possible ways of defining Cesaro means of functions of two
variables, we will use $ \si_{n, m}(f) ,$ the convolution of $f$ with $K_n\otimes K_m,$ where $K_n$ is the one-variable Fejer kernel.
However, there are other summability methods of functions of two variables which could be employed as well. (\cite{Zyg})

If $\al$ is a positive irrational number, we define $\sA_{\al}$ to be the set of continuous functions $f: \bbT^2 \to \bbC$ with the property that
\[ \hat{f}(m, n) = 0 \text{ whenever } m + \al n < 0 \]
By the continuity of the Fourier transform, $\sA_{\al}$ is a Banach space, and since the product of two functions in $\sA_{\al}$ again lies in $\sA_{\al},$ it is a Banach algebra under the norm
$||f|| = \sup_{(z, w) \in \bbT^2} |f(z, w)|.$

As a norm-closed subalgebra of the C$^*$-algebra $C(\bbT^2),\ \sA_{\al}$ is  a commutative operator algebra. It further falls in the category of uniform algebras, as
a subalgebra of $C(\bbT^2)$ it separates the points of $\bbT^2.$

For $f \in C(\bbT^2),$ let $f^*$ denote the adjoint, i.e., $f^* = \bar{f},$ the complex conjugate.  Note that $\sA_{\al}$ is \emph{antisymmetric}, that is, $\sA_{\al}\cap \sA_{\al}^*
= \{ \bbC \cdot 1\}.$

The \emph{characters} of the group $\bbT^2$ will be denoted by $\chi_{m, n}, $ where
\[ \chi_{m, n}(z, w) = z^m w^n,\ (z, w) \in \bbT^2. \]
The characters $\chi_{m, n} $ for which
$m + \al n \geq 0$ belong to $\sA_{\al},$ and linear combinations of characters in $\sA_{\al}$ are dense.

Note that $\sA_{\al}$ is a \emph{Dirichlet} algebra; that is $\sA + \sA^* $ is dense in $C(\bbT^2).$ This is clear, since $\sA + \sA^* $ contains all the characters of $\bbT^2.$

Recall that a uniform subalgebra $\sA$ of $C(X)$ is called \emph{maximal} if for any norm-closed algebra $\sB$ satisfying $\sA \subset \sB \subset C(X), $ then either
$\sB = \sA$ or $\sB = C(X).$ It is known that $\sA_{\al}$ is a maximal subalgebra of $C(\bbT^2).$ (\cite{Gam69})

Central to our work are results relating to invertible, continuous functions on compact Hausdorff spaces. The next
result is standard.
\begin{lemma} \label{l:invertonT}
Every invertible function $f$ in $C(\bbT)$ has the form $f(z) = z^n \exp(g(z))$ for $n \in \bbZ$ and
 some function $g \in C(\bbT).$ The integer $n$ is uniquely determined.
 \end{lemma}

 \begin{proof} This is Lemma 3.5.14 of \cite{Mur90}.
 \end{proof}

 We will need the analogous result for $\bbT^2 .$  For lack of a convenient reference, we provide an elementary proof.

 \begin{lemma} \label{l:invertibleonT2}
 Every invertible function $f$ in $C(\bbT^2)$ has the form
 \[ f = \chi_{m, n}\exp(g), \text{ for some } g \in  C(\bbT^2) .\]
 Furthermore the character $\chi_{m, n}$ is uniquely determined.
 \end{lemma}

 \begin{proof} Let $c = f(1, 1).$ If $\frac{1}{c} f$ has the desired form, then so does $f.$  Thus we may assume that $f(1, 1) = 1.$

 Given $f$ invertible in $C(\bbT^2),$ let $f_1(z) = f(z, 1)$ and $f_2(w) = f(1, w).$ Then
 $f_1,\ f_2 \in C(\bbT)$ and are invertible, so by Lemma~\ref{l:invertonT} there exist integers $m, \ n,$
 which are uniquely determined, and functions $g_1,\ g_2 \in C(\bbT)$ so that
 $f_1(z) = z^m\exp(g_1(z))$ and $f_2(w) = w^n\exp(g_2(w)).$

 Let $h(z, w) = z^{-m}\exp(-g_1(z)) w^{-n}\exp(-g_2(w)) f(z, w).$ Thus $h$ satisfies
 \[ h(z, 1) = 1 = h(1, w) \text{ for all } z,\ w \in \bbT .\]

 We claim that $h = \exp(k)$ for some $k \in C(\bbT^2).$ By Corollary~2.15 of \cite{Doug98},
 this is equivalent to showing that $h$ lies in the connected component of the constant function $1$
 in $C(\bbT^2).$

 Now let $w_0 \in \bbT \setminus\{1\},$ and let $t_0 \in (0, 2\pi)$ be such that $e^{it_0} = w_0.$
 Let $h_{w_0} \in C(\bbT)$ be the function $h_{w_0}(z) =  h(z, w_0).$ Observe that $h_{w_0}$ is
 path homotopic to the constant $1.$
 Define  $\ga(t)(\cdot) = h(\cdot, e^{it}),\ 0 \leq t \leq t_0$ and note that  $\ga(0) $ is the constant function
 $1,\ \ga(t_0) = h_{w_0},$ and that $\ga(t)(1) = h(1, e^t) = 1,\ t \in [0, t_0].$ We conclude that
 $h_{w_0}$ lies in the connected component of the identity of the invertibles in $C(\bbT),$ hence has the
 form $h_{w_0} = \exp(k_{w_0})$ for some $k_{w_0} \in C(\bbT),$ by Lemma~\ref{l:invertonT}
  Now $k_{w_0}$ is not unique, but as
 $h_{w_0}(1) = 1,$ we have that $k_{w_0}(1) \in 2\pi i\, \bbZ.$ If we specify that $k_{w_0}(1) = 0,$ then
 $k_{w_0}$ is unique.  If we do this for each $w \in \bbT,$ then the map $w \mapsto k_w$ is continuous.

 Now define $k(z, w) = k_z(w).$ Then $h = \exp(k),$ as desired.
\end{proof}


\subsection{The Gelfand space of $\sA_{\al}$} \label{ss:Gelfand}
The set of bounded, multiplicative linear functionals on a commutative Banach algebra $\sA$
 can be topologized with the weak* topology, making it a compact metric space, which is called the maximal ideal space, or
Gelfand space of $\sA,$ which we will denote $\Delta(\sA).$ In case $\sA = \sA_{\al},$ we will write
$\Delta(\sA_{\al})$ as  $\Delta_{\al}.$

Let $\bar{\bbD}$ be the closed unit disc $\{ z \in \bbC: |z| \leq 1 \}$ and $\bar{\bbD}^2 = \bar{\bbD}\times \bar{\bbD}$ the bidisc.
Let $\sB$ be the bidisc algebra, that is, the algebra of functions continuous on the bidisc and bianalytic in the interior.
As the Banach algebra satisfies
\[ \sB \subset \sA_{\al} \subset C(\bbT^2) \]
it follows that the Gelfand spaces satisfy the reverse inclusions:
\[ \Delta(C(\bbT^2))  \subset \Delta(\sA_{\al}) \subset(\Delta(\sB)) \]
In other words,
\[ \bbT^2 \subset \Delta_{\al} \subset \bar{\bbD}^2 \]

In fact, the Gelfand space $\Delta_{\al}$ has been characterized as follows:

\[ \Delta_{\al} = \{ (z, w) \in \bar{\bbD}^2: |w| = |z|^{\al}\} \]
See \cite{Ency}.

If $\sA$ is a uniform algebra on $X,$ we say that
$\Theta_1,\ \Theta_2 \in \Delta(\sA),$ \emph{belong to the same part of} $\Delta(\sA)$ if
\[ || \Theta_1 - \Theta_2|| < 2 \]
where the norm is the norm in the dual space of $\sA,$ that is
\[ || \Theta_1 - \Theta_2|| = \sup\{ |\Theta_1(f) - \Theta_2(f)|: f \in \sA, ||f|| = 1 \} .\]

For the algebra $\sA_{\al},$ it is known that (\cite{Wer61})
\begin{enumerate}
\item Each point $(z_0, w_0)$ in (the Shilov boundary) $\bbT^2$ is in a singleton part;
\item The point $(0, 0)$ is in a singleton part.
\end{enumerate}

\begin{remark} \label{r:parts}
From the observations above it follows that
the set $\{(z, w) \in \Delta_{\al}: \ 0 < |z| < 1\}$ is the union of parts.  We have proved that no part in this set is a singleton, though we have not included the proof here,
as we believe this fact is known. (A proof can be found in \cite{San}, section 3.3.)
\end{remark}


\section{Isometric Isomorphisms of the algebras $\sA_{\al}$}

\subsection{Method of rational approximation} \label{ss:ratapprox}

Suppose $f \in \sA_{\al}$ has a Fourier series with only finitely many terms, so
\[ f = \sum_{k=1}^N c_k\, \chi_{m_k, n_k} .\]

Consider the interval $I = \{ t \geq 0 : m_k + t n_k \geq 0,\ 1 \leq k \leq N \}.$ This is an interval containing
$\al$ in its interior.  Let $p,\, q$ be positive integers, such that $p/q \in I.$
Furthermore, since the map $\{ (m_k, n_k): 1 \leq k \leq N\} \mapsto m_k + \al n_k$ is one-to-one, one can
choose  $p/q$ sufficiently close to $\al$ so that the map $\{ (m_k, n_k): 1 \leq k \leq N\} \mapsto
m_k + (p/q)n_k$ is one-to-one.  Thus, if we define the function

\[ F(\zeta) = f(\zeta^q, \zeta^p) \quad (\dag) \]
then $F$ is a polynomial in $\zeta$ with $N$ (non-zero) terms.

Let $p_n, \ q_n, \ n \in \bbN$ be sequences of positive integers such that $p_n/q_n$ converges to $\al.$  Define a sequence of measures $\mu_n$ on $C(\bbT^2)$ by
\[ \mu_n(f) := \int f \, d\mu_n := \frac{1}{2\pi}\int_{-\pi}^{\pi} f(e^{i q_n \theta}, e^{i p_n\theta}) \, d\theta .\]
Let $\mu$ denote normalized Lebesgue measure on $\bbT^2.$

\begin{lemma} \label{l:convmeasures}
The sequence $\{ \mu_n\}$ converges in the weak $*$-topology to $\mu.$
\end{lemma}

\begin{proof}
First we show that for any character $\chi, \lim_n \int \chi \, d\mu_n = \int \chi \, d\mu .$
Indeed, if $\chi = \chi_{0, 0} = 1,$ then $\int \chi \, d\mu_n = 1 = \int \chi \, d\mu,$ since the measures are all positive of mass $1.$  If $\chi = \chi_{m, n}$ with
$(m, n) \neq (0, 0),$ then $m + \al n \neq 0, $ so that $m + (p_k/q_k) n \neq 0,$ for $k$ sufficiently large, hence $mq_k + np_k \neq 0.$ Then
\[ \int \chi \, d\mu_k = \frac{1}{2\pi} \int_{-\pi}^{\pi} e^{i(mq_k + np_k)\theta} \, d\theta = 0 \text{ and } \int \chi\, d\mu = 0 . \]

Thus the desired result holds for any $f \in C(\bbT^2)$ which is a finite linear combination of characters, which is a dense subalgebra of $C(\bbT^2).$

Now by the weak $*$-compactness of unit ball in the dual of $C(\bbT^2),$ there is a subnet of $\mu_n$ which converges, and by the metrizability of the dual space, the subnet
can be taken as a subsequence.  By re-labeling, we may denote the convergent subsequence by $\{ \mu_n\}.$

Next let $f \in C(\bbT^2),$
and $\ep > 0.$  Let $f_1$ be a Cesaro mean of $f$ such that $||f - f_1|| < \ep.$ Then
\begin{align*}
|\lim_{n \to \infty} \mu_n(f) - \mu(f)| &\leq |\lim_{n \to \infty} \mu_n(f - f_1)| + |\mu(f - f_1)| \\
	&\leq \lim_{n \to \infty} \mu_n(\ep) + \ep \\
	&\leq 2\ep
\end{align*}

Now, the same argument shows that any subsequence of the original sequence $\{ \mu_n\}$ in turn has a subsequence converging to $\mu.$  Thus $\{ \mu_n\}$
converges weak $*$ to $\mu.$
\end{proof}

\begin{lemma} \label{l:factorization}
Let $f \in \sA_{\al}$ and suppose $f = \chi_{m, n} \exp(g)$ for some character $\chi_{m, n}$ and some $g \in C(\bbT^2).$ Then both $\chi_{m, n} $ and $\exp(g) $ belong
to $\sA_{\al},$ and the extension of $\exp(g)$ to the Gelfand space $\Delta_{\al}$  does not vanish at the point $(0, 0).$
\end{lemma}

\begin{proof} First assume that $f$ is a finite linear combination of characters.  By the method of rational approximation, we can find integers $p, q$ with $p/q $ sufficiently close to
$\al$ so that the function $F(\zeta) = f(\zeta^q, \zeta^p) \in \sA(\bbD).$  Since $f$ is invertible on $\bbT^2,$ it follows $F$ is invertible on $\bbT,$ so by Lemma~\ref{l:invertonT}
$F$ has the form $F(\zeta) = \zeta^N \exp(G),$ with $N \geq 0.$  On the other hand, setting $N_1 = mq + np, $ we have $F(\zeta) = \zeta^{N_1} \exp(G_1),$ with $G_1(\zeta) = g(\zeta^q, \zeta^p).$
By the uniqueness assertion of the Lemma, $N_1 = N,$ and hence $G - G_1 \in 2\pi \bbZ.$
Now if $B$ is the Blaschke factor (which is the inner factor)
 in the inner-outer factorization of $F,$ then $B(\zeta) = \zeta^N(\frac{\zeta-a_1}{1 - \bar{a}_1 \zeta}\cdots \frac{\zeta-a_M}{1 - \bar{a}_M \zeta})$ where  $0 < |a_t| < 1,\ 1 \leq t \leq M.$
Furthermore, each factor in the Blaschke product belongs to the disc algebra, in particular the first factor.  Also, $\exp(G(\zeta)) =
 (\frac{\zeta-a_1}{1 - \bar{a}_1 \zeta}\cdots \frac{\zeta-a_M}{1 - \bar{a}_M \zeta}) F_0(\zeta),$
where $F_0$ is the outer factor.  Since these are in the disc algebra, so is $\exp(G).$ Note that since $N$ is the order of the zero of $F$ at the origin, $\frac{F(\zeta)}{\zeta^N}$
is non-zero at the origin.  So, the extension of $\exp(G)$ to the disc does not vanish at the origin.

Next we claim that both $\chi_{m, n}$ and $\exp(g) \in \sA_{\al}.$ Let $p_k, \ q_k $ be positive integers such that $\{\frac{p_k}{q_k}\}$ converges to $\al,$ and let $a, b$ be integers such that
$a + \al b \geq 0.$  Let $F_k(\zeta) = f(\zeta^{q_k}, \zeta^{p_k}) = \zeta^{N_k}\exp(G_k(\zeta)) .$
Then
\[ \int_{\bbT^2} \chi_{a, b}(z, w) \exp(g(z, w)) \, d\mu = \lim_k \int_{\bbT}\zeta^{aq_k + bp_k}\exp(G_k(\zeta^{q_k}, \zeta^{p_k}))\, d\theta = 0,\]
$\zeta = e^{i\theta}$, for $k$ sufficiently large, as $\exp(G_k) $ is in the disc algebra.

A similar argument shows that $\chi_{m, n} \in \sA_{\al}.$

 Since the polynomial $\exp(G_k)$ is nonzero at the origin, it contains a nonzero multiple of the constant function as a Fourier coefficient, hence the same is true of $\exp(g).$

Now for the general case, where  $f = \chi_{m, n} \exp(g) \in \sA_{\al},$ we can apply the above argument to the Cesaro means of $f.$  Note functions in the set
$\chi_{m, n} \exp(C(\bbT^2))$ constitute one of the connected components of the invertible functions in $C(\bbT^2),$ and in particular, this set is open, so that any Cesaro mean
sufficiently close to $f$ has this form. Thus considering the Cesaro means, we obtain a sequence of functions $\exp(g_n)$ converging
to $\exp(g),$ such that $\exp(g_n) \in \sA_{\al}$ for sufficiently large $n.$  Thus $\exp(g) \in \sA_{\al}.$

Finally, the extension of the function $\exp(g)$ to the Gelfand space $\Delta_{\al}$ is non-zero at the point $(0, 0).$ This is due to the nature of the Cesaro approximations: the
non-zero Fourier coefficients  of the Cesaro approximants is a subset of the non-zero Fourier coefficients of the function $\exp(g).$  Since all of the Cesaro approximants $\exp(g_n)$
contain
a non-zero multiple of the constant character, at least for sufficiently large $n, $ the $(0, 0)$ Fourier coefficient of $\exp(g)$ is non-zero.
\end{proof}

\begin{lemma} \label{l:sqroot}
Suppose $g \in C(\bbT^2)$ is such that $\exp(g) \in \sA_{\al},$ and the extension of $\exp(g)$ to $\Delta_{\al}$ is never zero.  Then $\exp(g/2) \in \sA_{\al}.$
\end{lemma}

\begin{proof} As in the previous lemma, we begin by assuming that $f = \exp(g) \in \sA_{\al} $ is expressible as a finite linear combination of characters. In that case, $f$ extends to a function on a subset $S$
of the closed bidisc, $S = \{ (z, w): |w| = |z|^t,\ a \leq t \leq b\}$ where $ 0 < a < \al < b.$  Since $f$ is nonzero on $\Delta_{\al}$ and uniformly continuous on $S,$ we may assume that $f$ is nonzero on $S,$
possibly by replacing $[a, b]$ by a smaller interval.

Suppose that $p, q$ are positive integers with $ a < p/q < b,$ and set  $F(\zeta) = f(\zeta^q, \zeta^p) ,$ so $F$ is in the disc algebra.  If for some $|\zeta_0| < 1,\ F(\zeta_0) = 0,$ then
$f(z_0, w_0) = 0,$ where $\zeta_0^q = z_0,\ \zeta_0^p = w_0.$ But then $|w_0| = r_0^p = |z_0| ^{\frac{p}{q}}, \  |\zeta_0| = r_0.$  Since $p/q \in (a, b),$ it follows that $(z_0, w_0) \in S,$
and hence $ 0 = F(\zeta_0) = f(z_0, w_0),$ a contradiction.

It follows that $F(\zeta)$ is outer.  Then, again by factorization, since $F(\zeta) = \exp(G), $ we obtain that $\exp(G/2) $ is in the disc algebra. (\cite{Hof62})

To see that $\exp(g/2) \in \sA_{\al},$ let $\frac{p_k}{q_k}$ be a sequence of fractions converging to $\al.$  Set $F_k(\zeta) = \exp(G_k(\zeta)).$  By the argument above, we have that
$\exp(G_k/2)$ is in the disc algebra, at least for sufficiently large $k.$  Then, for any $(m, n) \in \bbZ^2$ with $m + \al n \geq 0,$ we have
\[ \int_{\bbT^2} \chi_{m, n}(z, w) \exp(\frac{g}{2}(z, w)) \, d\mu = \lim_k \int_{\bbT} \zeta^{mq_k + np_k} \exp(\frac{G_k}{2}(\zeta)) \, d\theta = 0,  \]
$ \zeta = e^{i\theta}.$

For the general case, where $ f = \exp(g) \in \sA_{\al},$ apply the argument above to the Cesaro means $f_n$ of $f$ to get $f_n = \exp(g_n)$ with $\exp(g_n/2) \in \sA_{\al}.$
Thus $\exp(g_n/2)$  converges to $\exp(g/2),$ so that it is also in $\sA_{\al}.$
\end{proof}

\begin{proposition} \label{p:exponents}
Suppose $\chi_{m, n}$ is a character, and $g \in C(\bbT^2)$ are such that the function $f = \chi_{m, n} \exp(g) \in \sA_{\al},$  and $f$ does not vanish on $\Delta_{\al}\setminus \{(0, 0)\}.$
Then $\chi_{m, n}$ and $g $ lie in $\sA_{\al}.$
Furthermore, if $|\exp(g)| = 1$ on $\bbT^2,$ then $g$ is constant.
\end{proposition}

\begin{proof} Lemma~\ref{l:factorization} shows that $\chi_{m, n} \in \sA_{\al},$ and $\exp(g) \in \sA_{\al}$ does not vanish on $\Delta_{\al}.$
Then, by Lemma~\ref{l:factorization} and repeated application of Lemma~\ref{l:sqroot}
we obtain that
\[ \exp(g),\ \exp(g/2),\  \dots , \exp(g/2^k), \dots \in \sA_{\al}. \]

Now if $t_k$ is any sequence of positive reals decreasing to $0,$ then
\[ g = \lim_{k \to \infty} \frac{\exp(t_k g) - 1}{t_k} \]
where the convergence is in norm.  Applying this with $t_k = \frac{1}{2^k},$ we obtain the desired result.

If $|\exp(g)| = 1$ on $\bbT^2,$ then $\exp(-g) = \exp(g^*) = \exp(g)^* \in \sA_{\al} \cap \sA_{\al}^* = \{\bbC \cdot 1 \},$
so, $g$ is constant.
\end{proof}




\subsection{Admissible homeomorphisms}
\begin{lemma} \label{l:automorphism}
Let $(a, b) \in \bbT^2.$  The map
\[\ga: \sA_{\al} \to \sA_{\al},\ \ga(f)(z, w) = f(az, bw) \]
 is an isometric automorphism of $\sA_{\al}.$
\end{lemma}

\begin{proof} Clearly $\ga$, or more precisely, the extension of $\ga$ to $C(\bbT^2)$, is an isometric
automorphism of $C(\bbT^2).$ Furthermore, for any $f \in \sA_{\al},$ a character $\chi_{m, n}$ appears as
a non-zero Fourier coefficient of $f$ if and only if it appears in $\ga(f).$  Thus, $\ga$ maps $\sA_{\al}$
to itself.
\end{proof}

\begin{definition} \label{d:admissible}
Let $\sA,\ \sB$ be  uniform algebras with Gelfand spaces $\Delta(\sA),$\\
$\Delta(\sB)$ respectively.
  We will say that a homeomorphism
$\vpi: \Delta(\sB) \to \Delta(\sA) $ is \emph{admissible} if $f\circ \vpi \in \sA$ whenever $f \in \sB.$
\end{definition}

By Lemma~\ref{l:automorphism}, the homeomorphism of $\Delta_{\al}$ given by $(z, w) \mapsto (az, bw)$
 is admissible for any $(a, b) \in \bbT^2.$

If $\Phi: \sA_{\al} \to \sA_{\be}$ is an algebraic isomorphism, there is a weak*
 homeomorphism $\vpi: \Delta_{\be} \to
\Delta_{\al}$ defined by $f(\vpi(z, w)) = \Phi(f)(z, w).$ In other words, $\vpi$ is admissible
(Definition~\ref{d:admissible}).  However, if $\Phi$ is an isometric isomorphism, then more is true: $\vpi$
maps the Shilov boundary of $\Delta_{\be}$ to the Shilov boundary of $\Delta_{\al},$ and also maps
parts of $\Delta_{\be}$ to parts of $\Delta_{\al}.$  Thus, $\vpi$ maps $\bbT^2$ to itself, maps
the singleton part $\{ (0, 0)\} \in \Delta_{\be}$ to the corresponding part in $\Delta_{\al},$ and the set
$\{ (z, w) : |w| = |z|^{\be}, 0 < |z| < 1\} \in \Delta_{\be}$ to the corresponding set
$\{ (z, w): |w| = |z|^{\al}, 0 < |z| < 1 \} \in \Delta_{\al}.$ (See Remark~\ref{r:parts}.)

It should be noted that while these conditions are necessary conditions on $\vpi,$ it does not mean that
a homeomorphism satisfying these conditions is admissible.  Indeed, we will see that there are non-admissible
homeomorphisms satisfying these conditions.

\begin{theorem} \label{t:iso}
$\Phi: \sA_{\al} \to \sA_{\be}$ is an isometric isomorphism if and only if $\Phi$ has form
\mbox{$\Phi(f) = f\circ \vpi,$}
where $\vpi: \Delta_{\be} \to \Delta_{\al}$ is of the form
\[ (z, w) \mapsto (c_1z^{m_1}w^{n_1}, c_2z^{m_2} w^{n_2}) \]
where $c_j$ are unimodular constants, $j = 1, 2,$ and the matrix
\[ A = \left[ \begin{matrix}
m_1 & n_1 \\
m_2  & n_2
\end{matrix} \right] \in GL(2, \bbZ)\]
satisfies $m_1+\be n_1>0$ and $m_2 + \be n_2 = \al(m_1 + \be n_1).$
\end{theorem}

\begin{proof}
Assume that $\Phi : \sA_\al \to \sA_\be$ is an isometric isomorphism. Then there is a homeomorphism $\vpi: \Delta_{\be} \to
\Delta_{\al}$ defined by $f(\vpi(z, w)) = \Phi(f)(z, w).$ Recall that the characters $\chi_{1, 0},\ \chi_{0, 1}$ are the coordinate functions
\[ \chi_{1, 0}(z, w) = z, \quad \chi_{0, 1} (z, w) = w. \]
Set $f_1 = \Phi(\chi_{1, 0}) = \chi_{1, 0} \circ \vpi,\ $ and similarly define $f_2,$ with $\chi_{0, 1}$ in place of $\chi_{1, 0}$,
so that $\vpi(z, w) = (f_1(z, w), f_2(z, w)).$

Then $f_j \in \sA_{\be}$ and since the homeomorphism $\vpi$ maps $\Delta_{\be} \to \Delta_{\al}$ and in
particular maps the Shilov boundary of $\Delta_{\be}$ to the Shilov boundary of $\Delta_{\al},$ this implies that
$|f_j| = 1 $ on $\bbT^2,\ j = 1, 2.$ Furthermore $\vpi$ maps the part $(0, 0) \in \Delta_{\be}$ to
the part $(0, 0) \in \Delta_{\al}.$  This implies that $f_j(z, w) = 0$ if and only if $(z, w) = (0, 0).$

Now since $f_j$ is invertible on $\bbT^2,$ by Lemma~\ref{l:invertibleonT2}
it has the form $f_j = \chi_{m_j, n_j} \exp(g_j),\ j = 1, 2.$
By Proposition~\ref{p:exponents}, $\exp(g_j) $ is constant, say equal to $c_j,$ with $|c_j| = 1.$
Since $f_j = c_j \chi_{m_j, n_j} \in \sA_{\be},$ we have that $ m_j + \be n_j \geq 0$.
And clearly $m_j + \be n_j > 0,$ for $f_j$ cannot be constant.

Since
\[ \vpi(z, w) = (f_1(z, w), f_2(z, w)) \in \Delta_{\al} \]
we have that
\begin{align*}
|f_2(z, w)| &= |f_1(z, w)|^{\al} \\
|\chi_{m_2, n_2}(z, w)|  &= |\chi_{m_1, n_1}(z, w)|^{\al}  \\
|z^{m_2} w^{n_2}| &= |z^{m_1} w^{n_1}|^{\al}   \\
|z|^{m_2} |z|^{\be n_2}  &= (|z|^{m_1} |z|^{\be n_1})^{\al}
\end{align*}
since $ |w| = |z|^{\be}$ in $\Delta_{\be}.$
Hence $m_2 + \be n_2 = \al(m_1 + \be n_1).$

Furthermore, since $\vpi$ is invertible, the map
\[ \bbT^2 \to \bbT^2,\ (z, w) \mapsto (z^{m_1}w^{n_1}, z^{m_2} w^{n_2}) \]
is invertible, so that the matrix
\[ A = \left[ \begin{matrix}
m_1 & n_1 \\
m_2  & n_2
\end{matrix} \right] \in GL(2, \bbZ).\]

Conversely, assume that $\Phi$ has form in the assumption. By composing $\Phi$ with the map in Lemma~\ref{l:automorphism}, we can assume that $c_1=1=c_2$. Let $m,n \in \bbZ$. Since $m_2+\be n_2=\al(m_1+\be n_1)$, we have $(mm_1+nm_2)+\be (mn_1+nm_2)=(m+\al n)(m_1+\be n_1)$. Moreover, since $m_1+\be n_1>0$, we get the necessary and sufficient condition :
\begin{align*}
m+\al n\geq 0 \text{ if and only if } (m+\al n)(m_1+\be n_1)\geq 0.
\end{align*}
Let $\chi_{m,n}$ be a character of $\sA_\al$. Then
\begin{align*}
(\chi_{m,n}\circ \vpi) (z,w)=\chi_{m,n}(\vpi(z,w))&=\chi_{m,n}(z^{m_1}w^{n_1},z^{m_2}w^{n_2})\\
&=z^{mm_1+nm_2}w^{mn_1+nm_2}\\
&=\chi_{mm_1+nm_2,mn_1+nm_2}(z,w)
\end{align*}
Since $\vpi$ is a bijection on $\bbT^2$, $\Phi$ is isometric. Hence $\Phi : \sA_\al \to \sA_\be$ is an isometric isomorphism.
\end{proof}

\begin{example} \label{e:inadmissible} Let $\al > \be > 0$ be irrationals.
Define a homeomorphism $\psi: \Delta_{\be} \to \Delta_{\al} $ as follows $\psi(z, w) = (z', w')$ where $z' = z, \ w' = r^{\al - \be}w,$ where $r = |z|.$
Then $\psi$ is a continuous bijection from one compact Hausdorff space to another, hence a homeomorphism.  Furthermore, it maps the Shilov
boundary of $\Delta_{\al}$ onto the Shilov boundary  of $\Delta_{\be},$ and the part $(0, 0)$ to $(0, 0).$

Observe that $\psi$ is the identity map on $\bbT^2.$ Now the algebras $\sA_{\al},\ \sA_{\be}$ are
determined by their restrictions to $\bbT^2.$ Since $\psi$ is the identity map on $\bbT^2,$ it could only be admissible if the two algebras were equal. Since $\al \neq \be,$ the two
algebras are not the same.  Thus $\psi$ is an inadmissible homeomorphism.
\end{example}

If $\al$ is an irrational, let $G_{\al}$ denote the dense additive subgroup $\{m + \al n: \ m, n \in \bbZ\} \subset \bbR.$
\begin{corollary} \label{c:orderedgroup}
Let $\al,\ \be$ be positive irrationals.  Then the algebras $\sA_{\al},\ \sA_{\be}$ are isometrically isomorphic if and only if the groups $G_{\al},\ G_{\be}$ are
order isomorphic.  That is, if and only if there is a group isomorphism from $G_{\al}$ to $G_{\be}$ which maps positive elements of $G_{\al}$ to positive elements of $G_{\be}.$
\end{corollary}

\begin{proof}
We use the fact that $G_\al$ and $G_\be$ are isomorphic to $\bbZ^2$ and Theorem~\ref{t:iso}.
\end{proof}


\section{Isometric Automorphisms of the $\sA_{\al}$}

Throughout, automorphisms always mean isometric automorphisms. In this section, we will investigate the automorphism group $\Aut(\sA_\al)$ of the uniform algebra $\sA_\al$. By Theorem~\ref{t:iso} from the previous section to the case $\be=\al$, it follows immediately that if $\al$ is not quadratic irrational, then $\Aut(\sA_\al) \cong \bbT^2$. Therefore, we will focus on the case when $\al$ is a positive quadratic irrational number. For such an $\al$, we show that $\Aut(\sA_\al)$ is a semidirect product of $\bbT^2$ and $\bbZ$.\\
\indent First, we would like to introduce some notations that we will use throughout this section. For $f \in C(\bbT^2)$, $(c_1,c_2) \in \bbT^2$, and $A = \left[ \begin{matrix}
a & b \\
c & d
\end{matrix} \right] \in GL(2, \bbZ)$,
let $\pi((c_1,c_2)), \pi(A): C(\bbT^2) \to C(\bbT^2)$ be defined by
\[
\pi((c_1,c_2))(f)(z, w) = f(c_{1}z, c_{2}w),\text{ and}\]
\[\pi(A)(f)= f\circ \vpi,\]
where $\vpi: \bbT^2 \to \bbT^2 \text{ is of the form } (z, w) \mapsto (z^{a}w^{b}, z^{c} w^{d})$. \\
\indent Note that we can view $\pi(\bold{c})$ and $\pi(A)$ as the restrictions of $\pi(\bold{c})$ and $\pi(A)$ to $\sA_\al$ where $\bold{c} \in \bbT^2$ and $A \in GL(2,\bbZ)$.\\
\indent To find the automorphism group of $\sA_\al$, when $\al$ is positive quadratic irrational, we begin by using Theorem~\ref{t:iso} in the case $\be=\al$ to get the following Corollary:

\begin{corollary}\label{c:auto}
Let $A=\begin{bmatrix}
m_1 & n_1 \\
m_2 & n_2
\end{bmatrix} \in GL(2,\bbZ)$.Then
$\pi(A)$ is an automorphism of $\sA_\al$ if and only if the matrix
$A$ satisfies $m_1+\al n_1>0$ and
\begin{equation*}
n_1\al^2+(m_1-n_2)\al-m_2=0.
\end{equation*}
\end{corollary}

\begin{remark} For any matrix  $A=\begin{bmatrix}
m_1 & n_1 \\
m_2 & n_2
\end{bmatrix} \in GL(2,\bbZ)$ that satisfies conditions in Corollary~\ref{c:auto}, $A\begin{bmatrix}
1 \\
\al
\end{bmatrix}=(m_1+\al n_1)\begin{bmatrix}
1 \\
\al
\end{bmatrix}$, i.e., $\begin{bmatrix}
1 \\
\al
\end{bmatrix}$ is an eigenvector of $A$ with positive eigenvalue $m_1+\al n_1$.
\end{remark}

\begin{lemma}\label{l:existsauto}
If $\alpha$ is a positive quadratic irrational number, then there exists a non-identity matrix $A \in SL(2,\bbZ)$ such that $\pi(A)$ is an automorphism of $\sA_{\al}$.
\end{lemma}

\begin{proof}
This follows from Lemma 8.12 of \cite{Kar13}.
\end{proof}

\begin{lemma}\label{l:commute}
Let $A_1,A_2$ be matrices in $GL(2, \bbZ)$ and $\al$ a positive quadratic irrational number. If $\pi(A_1)$ and $\pi(A_2)$ are automorphisms of $\sA_\al$, then $A_1A_2=A_2A_1$.
\end{lemma}

\begin{proof} Let  $A_1,A_2 \in GL(2, \bbZ)$ be such that $\pi(A_1)$ and $\pi(A_2)$ are automorphisms of $\sA_\al$. Let $\sH=\{A \in GL(2,\bbZ) : A\begin{bmatrix}
1 \\
\al
\end{bmatrix}=\lambda\begin{bmatrix}
1 \\
\al
\end{bmatrix},\text{ for some } \lambda>0\}$. Then $\sH$ is a subgroup of $GL(2,\bbZ)$. Let $\phi : \sH \to \bbR^+$ be defined by $\phi(A)=\lambda$. Then $\ker{\phi}=\{I\}$ Thus $\sH$ is isomorphic to a subgroup of $\bbR^+$. Hence $\sH$ is commutative. Since $A_1,A_2 \in \sH$, $A_1A_2=A_2A_1$.
\end{proof}

\begin{lemma}\label{l:nosquareroot}
Let $\alpha$ be a positive irrational number. Let
\begin{align*}
\sH=\{A \in SL(2,\bbZ) : A\begin{bmatrix}
1 \\
\al
\end{bmatrix}=\lambda\begin{bmatrix}
1 \\
\al
\end{bmatrix},\text{ for some } \lambda>0\}.
\end{align*}
If $A \in \sH$ with distinct eigenvalues and $n$ is an even number, then $A^n$ has no square root in $GL(2,\bbZ)\setminus SL(2,\bbZ)$.
\end{lemma}

\begin{proof}
Let $n$ be even and $\lambda$ an eigenvalue of $A$ corresponding to an eigenvector $\begin{bmatrix}
1 \\
\al
\end{bmatrix}$. Then $\frac{1}{\lambda}$ is the other eigenvalue of $A$. We write $A=P\begin{bmatrix}
\lambda & 0 \\
0 & \frac{1}{\lambda}
\end{bmatrix}P^{-1}$, where $P \in GL(2,\bbR)$ and the first column of $P$ is $\begin{bmatrix}
1 \\
\al
\end{bmatrix}$. Since $n$ is even, $n=2k$ for some $k \in \bbZ$. Assume that $A^n$ has a square root in $GL(2,\bbZ)\setminus SL(2,\bbZ)$. Then $P\begin{bmatrix}
\lambda^k & 0 \\
0 & -\frac{1}{\lambda^k}
\end{bmatrix}P^{-1} \in GL(2,\bbZ)$. Thus
\begin{align*}
P\begin{bmatrix}
1 & 0 \\
0 & -1
\end{bmatrix}P^{-1}=P\begin{bmatrix}
\lambda^k & 0 \\
0 & -\frac{1}{\lambda^k}
\end{bmatrix}P^{-1}A^{-k} \in GL(2,\bbZ).
\end{align*}
Note that the only matrix in $GL(2,\bbZ)$ that has an eigenvector $\begin{bmatrix}
1 \\
\al
\end{bmatrix}$ with eigenvalue $1$ is the identity matrix.
Since $P\begin{bmatrix}
1 & 0 \\
0 & -1
\end{bmatrix}P^{-1} \in GL(2,\bbZ)$ has an eigenvector $\begin{bmatrix}
1 \\
\al
\end{bmatrix}$ with eigenvalue $1$, $P\begin{bmatrix}
1 & 0 \\
0 & -1
\end{bmatrix}P^{-1}$ must be the identity matrix. This is a contradiction since $\operatorname{det}(P\begin{bmatrix}
1 & 0 \\
0 & -1
\end{bmatrix}P^{-1})=-1$. Hence  $A^n$ has no square root in $GL(2,\bbZ)\setminus SL(2,\bbZ)$.
\end{proof}

\begin{lemma}\label{l:dependent}
Let $A,B$ be matrices in $\mathfrak{sl}(2,\bbR)$. If $B \neq 0$ and $[A, B] = 0 $, then $\exists \, k \in \bbR$ such that $A=kB$.
\end{lemma}

\begin{proof}
Routine
\end{proof}

\begin{lemma}\label{l:singlegenerator}
If $\al$ is a positive quadratic irrational number, then there is a non-identity matrix $A_0 \in GL(2,\bbZ)$ such that for any matrix $A$ with $\pi(A)\in \Aut(\sA_\al)$, $\pi(A)$ is of the form $(\pi(A_0))^n$, for some $n \in \bbZ$.
\end{lemma}

\begin{proof}
By Lemma~\ref{l:existsauto}, there is a non-identity matrix $B\in SL(2,\bbZ)$ such that $\pi(B)$ is an automorphism of $\sA_\al$. Since $B \in SL(2,\bbR)$, there exist a non-zero matrix $S_1 \in \mathfrak{sl}(2,\bbR)$ such that $B=\exp(S_1)$. Let $\sH=\{A \in SL(2,\bbZ) : A\begin{bmatrix}
1 \\
\al
\end{bmatrix}=\lambda\begin{bmatrix}
1 \\
\al
\end{bmatrix},\text{ for some } \lambda>0\}$. Then $B \in \sH$.  Let $\sC=\{t \in \bbR : \exp(tS_1) \in \sH\}$. Then $\sC$ is a closed additive subgroup of $\bbR$. Let $t_0=\inf\{ t>0 : \exp(tS_1) \in \sH \}$. Then $t_0$ generates the group $\sC$. Let $A'=\exp(t_0S_1)$. Note that $A'$ cannot have a square root in $SL(2,\bbZ)$. If $A'$ has no square root in $GL(2, \bbZ)$, then let $A_0=A'$. If $A'$ has a square root in $GL(2,\bbZ)$, then let $A_0$ be the square root that has positive eigenvalue with respect to the eigenvector $\begin{bmatrix}
1 \\
\al
\end{bmatrix}$.\\
\indent Let $A \in GL(2,\bbZ)$ be such that $\pi(A) \in \Aut(\sA_\al)$. If $A \in SL(2,\bbZ)$, then there exists $S_2 \in \mathfrak{sl}(2,\bbR)$ such that $A=\exp(S_2)$. By Lemma~\ref{l:commute}, $BA=AB$. Thus $S_1S_2=S_2S_1$. By Lemma~\ref{l:dependent},  $\exists \, k \in \bbR$ so that $S_2=kS_1$. Since $A \in \sH$, $k \in \sC$. Then $\exists \, j \in \bbZ$ such that $k=jt_0$ Hence $A=(A')^{j}$. Thus $A=A_0^n$ for some $n \in \bbZ$. If $A \not\in SL(2,\bbZ)$, then $A^2 \in SL(2,\bbZ)$. By similar argument, we have $\exists \, n \in \bbZ$ such that $A^2=(A')^{n}$. By Lemma~\ref{l:nosquareroot}, $n$ must be odd.\\
\indent Claim: $A'$ has a square root in $GL(2,\bbZ)$.\\
Let $\lambda$ and $\lambda'$ be eigenvalues of $A$ and $A'$ corresponding to an eigenvector $\begin{bmatrix}
1 \\
\al
\end{bmatrix}$. By Lemma~\ref{l:commute}, $AA'=A'A$. Since both $A$ and $A'$ are diagonalizable and they commute, $A$ and $A'$ are simultaneously diagonalizable. Then we can write $A=P\begin{bmatrix}
\lambda & 0 \\
0 & -\frac{1}{\lambda}
\end{bmatrix}P^{-1}$ and $A'=P\begin{bmatrix}
\lambda' & 0 \\
0 & \frac{1}{\lambda'}
\end{bmatrix}P^{-1}$, where $P \in GL(2,\bbR)$  and the first column of $P$ is $\begin{bmatrix}
1 \\
\al
\end{bmatrix}$. Since $A^2=(A')^n$, we have $\lambda^2=(\lambda')^n$. Since $n$ is odd, $n=2k+1$ for some $k \in \bbZ$. Thus $\lambda=(\lambda')^{k+\frac{1}{2}}$. Hence
\begin{align*}
P\begin{bmatrix}
\sqrt{\lambda'} & 0 \\
0 & -\frac{1}{\sqrt{\lambda'}}
\end{bmatrix}P^{-1}&=P\begin{bmatrix}
(\lambda')^{-k}\lambda & 0 \\
0 & -\frac{1}{(\lambda')^{-k}\lambda}
\end{bmatrix}P^{-1}\\
&=P\begin{bmatrix}
\frac{1}{(\lambda')^k} & 0 \\
0 & (\lambda')^k
\end{bmatrix}P^{-1}P\begin{bmatrix}
\lambda & 0 \\
0 & -\frac{1}{\lambda}
\end{bmatrix}P^{-1}\\
&=(A')^{-k}A \in GL(2,\bbZ).
\end{align*}
In this case, $A_0$ is  the square root of $A'$ that has positive eigenvalue with respect to the eigenvector $\begin{bmatrix}
1 \\
\al
\end{bmatrix}$. Thus $A=A_0^{n}$. Hence  $\pi(A)=(\pi(A_0))^n$.
\end{proof}

Now, if $\al$ is positive quadratic irrational, we know from Theorem~\ref{t:iso} and  Lemma~\ref{l:singlegenerator} that $\Aut(\sA_\al)$ is the set $\{\pi(\bold{c})\pi(A)^k : \bold{c} \in \bbT^2, k \in \bbZ\}$ for some non-identity matrix $A \in GL(2, \bbZ)$.\\
\indent For $A=\begin{bmatrix}
a & b \\
c & d
\end{bmatrix} \in GL(2,\bbZ)$, let $\psi_A:\bbT^2 \to \bbT^2$ be defined by $\psi_A((c_1,c_2))=(c_1^ac_2^b,c_1^cc_2^d)$, for $(c_1,c_2) \in \bbT^2$. Note that $\Aut(\bbT^2)=\{\psi_A : A \in GL(2,\bbZ)\}$. For each $A \in GL(2,\bbZ)$, let $\bbT^2 \rtimes_{\psi_A} \bbZ$ denote the semidirect product of $\bbT^2$ and $\bbZ$, where the group multiplication of $\bbT^2 \rtimes_{\psi_A} \bbZ$ is given by $(\bold{c},m).(\bold{d},n)=(\bold{c}\psi_A^m(\bold{d}),m+n)$.

\begin{theorem} \label{t:semidirect}
Let $\al$ be a positive irrational.  If $\al$ is not a quadratic irrational, then $Aut(\sA_{\al}) \cong \bbT^2.$
If $\al$ is positive quadratic irrational, then $\Aut(\sA_\al) \cong \bbT^2 \rtimes_{\psi_A} \bbZ$ for some non-identity matrix $A \in GL(2,\bbZ)$.
\end{theorem}

\begin{proof}
If $\al$ is not a quadratic irrational, then from Theorem~\ref{t:iso} and Corollary~\ref{c:auto}, $Aut(\sA_{\al}) \cong \bbT^2.$

Assume $\al$ is positive quadratic irrational.
Then $\Aut(\sA_\al)=\{\pi(\bold{c})\pi(A)^k : \bold{c} \in \bbT^2, k \in \bbZ\}$ for some non-identity matrix
$A \in GL(2, \bbZ)$. Let $N_\al$ denote the subgroup generated by $\{\pi(\bold{c}) : \bold{c} \in \bbT^2\}$ and
 $<\pi(A)>$ denote the subgroup generated by $\pi(A)$.
First, we will show that $N_\al \, \unlhd \, \Aut (\sA_\al)$. Let $k\in\bbZ$ and $(c_1,c_2)\in \bbT^2$. Let $f \in A_\alpha$. Let $A^{k}=\begin{bmatrix}
a & b \\
c & d
\end{bmatrix}$ and $D = \det(A^k).$ Then we compute
\begin{align*}
\pi(A)^k\pi((c_1,c_2))\pi(A)^{-k}(f)(z,w)&=\pi((c_1,c_2))\pi(A)^{-k}(f)(z^aw^b,z^cw^d)\\
&=\pi(A)^{-k}(f)(c_1z^aw^b,c_2z^cw^d)\\
&=\pi((c_1^{\frac{d}{D}}c_2^{-\frac{b}{D}},c_1^{-\frac{c}{D}}c_2^{\frac{a}{D}}))(f)(z,w).
\end{align*}
Thus $\pi(A)^k\pi((c_1,c_2))\pi(A)^{-k} \in N_\al$. Hence $N_\al \unlhd \, \Aut (\sA_\al)$. Now we have that $\Aut (\sA_\al)=N_\al <\pi(A)>$ and $N_\al\cap <\pi(A)>=\{1\}$. Note that
\begin{align*}
\pi(\bold{c})\pi(A)^{k_1}\pi(\bold{d})\pi(A)^{k_2}&=\pi(\bold{c})(\pi(A)^{k_1}\pi(\bold{d})\pi(A)^{-k_1})\pi(A)^{k_1+k_2}\\
&=\pi(\bold{c})\pi(\psi_A^{-k_1}(\bold{d}))\pi(A)^{k_1+k_2},
\end{align*}
for $\bold{c},\bold{d}\in \bbT^2,k_1,k_2 \in \bbZ$. Thus $\Aut(\sA_\al) \cong \bbT^2 \rtimes_{\psi_A} \bbZ$.
\end{proof}

\begin{theorem}\label{t:autiso}
Let $\al, \beta$ be positive quadratic irrationals. If $\Aut (\sA_\al)\cong \bbT^2 \rtimes_{\psi_A} \bbZ$ and $\Aut (\sA_\be)\cong \bbT^2 \rtimes_{\psi_B} \bbZ$, then
$\Aut (\sA_\al)\cong \Aut (\sA_\be)$ if and only if $B=C^{-1}AC$ or $B^{-1}=C^{-1}AC$ for some $C\in GL(2,\bbZ)$.
\end{theorem}

\begin{proof}
First, we will show that there is no group homomorphism from $\bbT^2$ onto $\bbZ$. Assume that there is a group epimorphism $\phi:\bbT^2 \to \bbZ$. Then $\exists\, \bold{c} \in \bbT^2$ such that $\phi(\bold{c})=1$. Since $\bold{c}\in \bbT^2$, $\exists \, \bold{d} \in \bbT^2$ so that $\bold{d}^2=\bold{c}$. Thus $1=\phi(\bold{c})=\phi(\bold{d}.\bold{d})=\phi(\bold{d})+\phi(\bold{d})=2\phi(\bold{d})$ Hence $\phi(\bold{d})=\frac{1}{2}$, a contradiction. Therefore, there is no such an epimorphism. By Proposition 2.1 of \cite{ArzLafMin14}, $\bbT^2 \rtimes_{\psi_A} \bbZ \cong \bbT^2 \rtimes_{\psi_B} \bbZ$ if and only if $\psi_A$ is conjugate to $\psi_B$ or $\psi_B^{-1}$. Thus $\bbT^2 \rtimes_{\psi_A} \bbZ \cong \bbT^2 \rtimes_{\psi_B} \bbZ$ if and only if $A$ is conjugate to $B$ or $B^{-1}$, i.e., $\exists\, C \in GL(2,\bbZ)$ such that $B=C^{-1}AC$ or $B^{-1}=C^{-1}AC$.
\end{proof}

\begin{remark} \label{r:autiso}
Let $\al$ and $\be$ be positive quadratic irrationals. Then one can ask if
 $\Aut (\sA_\al)$ isometrically isomorphic to $\Aut (\sA_\be)$  implies $\sA_\al$  isometrically isomorphic to $\sA_\be$.   Let $A$  (resp., $B$) $ \in GL(2, \bbZ)$ be such that $\pi(A),\ \pi(\bbT^2)$ generate $\Aut(\sA_{\al})$
(resp., $\sA_{\be}$).  Let $C \in GL(2, \bbZ)$ be such that $B = C^{-1}A C.$  Then $C$ maps the eigenvector $[1, \be]^{\top}$ of $B$ to an eigenvector of $A$ with the same eigenvalue.  If $\det(A) = -1,$ and hence
$\det(B) = -1.$ then $C $ maps $[1, \be]^{\top} $ to a positive multiple of $[1, \al]^{\top}.$ In that case, we can apply Theorem~\ref{t:iso} to conclude that $\sA_{\al} \cong \sA_{\be}.$  However if $A,\ B \in SL(2. \bbZ),$
then we don't know if $C$ maps $[1, \be]^{\top}$ to a positive multiple of $[1, \al]^{\top}$ or the other eigenvector.

Conjecture: If $A \in SL(2, \bbZ)$ is such that $\pi(A), \ \pi(\bbT^2)$ generate $\Aut(\sA_{\al}),$ then $A$ is conjugate to $A^{-1}$ in $GL(2, \bbZ).$

If the conjecture is true, then we obtain a positive answer to the question raised above. We note that if $A \in SL(2, \bbZ)$ and  $\al = \sqrt{\frac{p}{q}}$ as in Proposition~\ref{p:pq}, then $A$ and $A^{-1}$ are conjugate.
\end{remark}


\section{Computation of the Automorphism Group of $\sA_\al$.}

In this section, we will find an explicit formula for any matrix $A \in GL(2,\bbZ)$ such that $\pi(A) \in \Aut(\sA_\al)$ by using the Pell's equations.\\
\indent Pell's equation is a Diophantine equation of the form
\begin{align*}
x^2-ny^2=1,
\end{align*}
where $n$ is a positive nonsquare integer. This equation is always solvable in integers and has the trivial solution with $x=1$ and $y=0$. Moreover, this equation always has nontrivial solutions. It is well-known that the set of solutions of this equation is given by
\begin{align*}
\{(1,0), (-1,0),(x_k,y_k),&(-x_k,y_k),(x_k,-y_k),(-x_k,-y_k)\,:\\
x_{k+1}&=x_1x_k+ny_1y_k,\\
y_{k+1}&=x_1y_k+y_1x_k,\\
&k=1,2,\ldots\},
\end{align*}
where $(x_1,y_1)$ is the fundamental solution of $x^2-ny^2=1$. The fundamental solution  of $x^2-ny^2=1$ is the pair $(x_1, y_1)$, $x_1$ is the smallest positive interger and $y_1$ is the positive integer that  satisfies
\begin{align*}
x_1^2-ny_1^2=1.
\end{align*}
Now, we look at the equation of the form
\begin{align*}
x^2-ny^2=-1,
\end{align*}
where $n$ is a positive nonsquare integer. This equation is called the negative Pell's equation. Note that, this equation may have no solutions in integers. There is a necessary but not sufficient condition of this equation to have integer solutions that all odd prime factors of $n$ must be congruent to $1$ modulo $4$. If this negative Pell's equation has solutions, then the set of all solutions is given by
\begin{align*}
\{(x_k',y_k'),&(-x_k',y_k'),(x_k',-y_k'),(-x_k',-y_k')\,:\\
x_{k+1}'&=(x_1'^2+ny_1'^2)x_k'+2nx_1'y_1'y_k',\\
y_{k+1}'&=(x_1'^2+ny_1'^2)y_k'+2x_1'y_1'x_k',\\
&k=1,2,\ldots\},
\end{align*}
where $(x_1',y_1')$ is the fundamental solution of $x^2-ny^2=-1$. Moreover, the fundamental solution $(x_1,y_1)$ of the Pell's equation $x^2-ny^2=1$ can be obtained from the fundamental solution $(x_1',y_1')$ of the negative Pell's equation
$x^2-ny^2=-1$ by
\begin{align*}
x_1=x_1'^2+ny_1'^2 \text{ and } y_1=2x_1'y_1'.
\end{align*}
Another equations that relates to our problems are the equations of the forms
\begin{align*}
x^2-ny^2=\pm4,
\end{align*}
where $n$ is a positive nonsquare integer. The equation $x^2-ny^2=4$ is always solvable over integers.  The set of solutions of this equation is given by
\begin{align*}
\{(2,0), (-2,0),(x_k,y_k),&(-x_k,y_k),(x_k,-y_k),(-x_k,-y_k)\,:\\
x_{k+1}&=\frac{1}{2}(x_1x_k+ny_1y_k),\\
y_{k+1}&=\frac{1}{2}(x_1y_k+y_1x_k),\\
&k=1,2,\ldots\},
\end{align*}
where $(x_1,y_1)$ is the fundamental solution of $x^2-ny^2=4$. The equation $x^2-ny^2=-4$ may not be solvable over integers. But if it has solutions, then the set of all solutions is
\begin{align*}
\{(x_k',y_k'),&(-x_k',y_k'),(x_k',-y_k'),(-x_k',-y_k')\,:\\
x_{k+1}'&=\frac{1}{4}((x_1'^2+ny_1'^2)x_k'+2nx_1'y_1'y_k'),\\
y_{k+1}'&=\frac{1}{4}((x_1'^2+ny_1'^2)y_k'+2x_1'y_1'x_k'),\\
&k=1,2,\ldots\},
\end{align*}
where $(x_1',y_1')$ is the fundamental solution of $x^2-ny^2=-4$.
Moreover, the fundamental solution $(x_1,y_1)$ of the equation $x^2-ny^2=4$ can be obtained from the fundamental solution $(x_1',y_1')$ of
$x^2-ny^2=-4$ by
\begin{align*}
x_1=\frac{1}{2}(x_1'^2+ny_1'^2) \text{ and } y_1=x_1'y_1'.
\end{align*}
 \indent Now, for each quadratic irrational $\al>0$, we want to examine the automorphism group $\Aut(\sA_\al)$ of $\sA_\al$. Note that for any positive quadratic irrational number $\alpha$, we can write $\al$ in one of following forms:\\
(1) $\al=\sqrt{\frac{p}{q}}$, $p,q \in \bbN$, and $\gcd(p,q)=1$,\\
(2) $\al=\frac{r}{s}+k\sqrt{\frac{p}{q}}$, $r \in \bbZ, p,q,s \in \bbN$, $k \in \{-1,1\}$, $\gcd(r,s)=1$, and $\gcd(p,q)=1$.\\
Note that if $\sqrt{\frac{p}{q}}$ is irrational, then $pq$ is nonsquare.

\begin{proposition}\label{p:pq}
Let $\alpha$ be a positive irrational number of the form $\sqrt{\frac{p}{q}}$, where $p,q \in \bbN$, and $\gcd(p,q)=1$.\\
(1) If the equation $x^2-pqy^2=-1$ is not solvable over the integers, then, for any $A \in GL(2,\bbZ)$,
$\pi(A)$ is an automorphism of $\sA_{\al}$ if and only if $A=\begin{bmatrix}
x_1 & qy_1 \\
py_1 & x_1
\end{bmatrix}^n$, for some $n \in \bbZ$, where $(x_1,y_1)$ is the fundamental solution of the Pell's equation $x^2-pqy^2=1$.\\
(2) If the equation $x^2-pqy^2=-1$ is solvable over the integers, then, for any $A \in GL(2,\bbZ)$,
$\pi(A)$ is an automorphism of $\sA_{\al}$ if and only if $A=\begin{bmatrix}
x_1' & qy_1' \\
py_1' & x_1'
\end{bmatrix}^n$, for some $n \in \bbZ$, where $(x_1',y_1')$ is the fundamental solution of the negative Pell's equation  $x^2-pqy^2=-1$.
\end{proposition}

\begin{proof}
Let $A=\begin{bmatrix}
a & b \\
c & d
\end{bmatrix} \in GL(2,\bbZ)$ be such that $a+b\sqrt{\frac{p}{q}}>0$ and
\begin{align*}
b\frac{p}{q}+(a-d)\sqrt{\frac{p}{q}}-c=0.
\end{align*}
Then $a=d$ and $c=\frac{p}{q}b$. So $A$ is of the form $\begin{bmatrix}
a & b \\
\frac{p}{q}b & a
\end{bmatrix}$.
Since $A \in GL(2,\bbZ)$ and $\gcd(p,q)=1$, $q|b$. So $\exists\, j \in \bbZ$ such that $b=qj$. Thus
\begin{align*}
A=\begin{bmatrix}
a & qj \\
pj & a
\end{bmatrix}.
\end{align*}
If $\operatorname{det}(A)=1$, then $a^2-pqj^2=1$. So $a^2=1+pqj^2>pqj^2=\frac{p}{q}b^2$. Thus $|a|>|b|\sqrt{\frac{p}{q}}$. Since $a+b\sqrt{\frac{p}{q}}>0$, $a$ must be positive. In this case, we have to find integers $a ,j$ such that $a>0$, $a^2-pqj^2=1$.\\
If $\operatorname{det}(A)=-1$, then $a^2-pqj^2=-1$. So $\frac{p}{q}b^2=pqj^2=1+a^2>a^2$. Thus $|b|\sqrt{\frac{p}{q}}>|a|$. Since $a+b\sqrt{\frac{p}{q}}>0$, $b$ must be positive, i.e., $j$ must be positive. In this case, we have to find integers $a ,j$ such that $j>0$, $a^2-pqj^2=-1$.\\
Case (1) :  The equation $x^2-pqy^2=-1$ is not solvable over the integers.\\
Then $\operatorname{det}(A)$ must be 1. Now, $A$ is of the form $\begin{bmatrix}
a & qj \\
pj & a
\end{bmatrix}$, where $a ,j \in \bbZ$ such that $a>0$ and $a^2-pqj^2=1$.
Let $(x_1,y_1)$ be the positive fundamental solution of the Pell's equation $x^2-pqy^2=1$. Then all nonnegative solutions $(x,y)$ are given by the set
\begin{align*}
\{(1,0), (x_n,y_n) : x_{n+1}=x_1x_n+pqy_1y_n, y_{n+1}=x_1y_n+y_1x_n, n=1,2,\ldots\}.
\end{align*}
By using mathematical induction and recurence relations of solutions, it is not hard to show that, for any $n \in \bbN$,
\begin{align*}
\begin{bmatrix}
x_1 & qy_1 \\
py_1 & x_1
\end{bmatrix}^{n}=\begin{bmatrix}
x_{n} & qy_{n} \\
py_{n} & x_{n}
\end{bmatrix}
\end{align*}
and
\begin{align*}
\begin{bmatrix}
x_1 & qy_1 \\
py_1 & x_1
\end{bmatrix}^{-n}=\begin{bmatrix}
x_{n} & q(-y_{n}) \\
p(-y_{n}) & x_{n}
\end{bmatrix}.
\end{align*}
Note that if $a=1$ and $j=0$, $A=\begin{bmatrix}
1 & 0 \\
0 & 1
\end{bmatrix}$.
Thus  $A=\begin{bmatrix}
x_1 & qy_1 \\
py_1 & x_1
\end{bmatrix}^n$, for some $n \in \bbZ$.\\
Case (2) :  The equation $x^2-pqy^2=-1$ is solvable over the integers.\\
Then $A$ must be of the form $\begin{bmatrix}
a & qj \\
pj & a
\end{bmatrix}$ where $a,j \in \bbZ$,
\begin{align*}
a>0 \text{ and } a^2-pqj^2=1
\end{align*}
or
\begin{align*}
j>0 \text{ and } a^2-pqj^2=-1.
\end{align*}
Let $(x_1',y_1')$ be the fundamental solution of the negative Pell's equation $x^2-pqy^2=-1$. Then $(x_1,y_1)=(x_1'^2+pqy_1'^2, 2x_1'y_1')$ is the fundamental solution of $x^2-pqy^2=1$. Also,  all positive solutions $(x,y)$ of $x^2-pqy^2=-1$ are given by the set
\begin{align*}
\{(x_n',y_n') : &x_{n+1}'=(x_1'^2+pqy_1'^2)x_n'+2pqx_1'y_1'y_n',\\
&y_{n+1}'=(x_1'^2+pqy_1'^2)y_n'+2x_1'y_1'x_n', n=1,2,\ldots\}
\end{align*}
and  all nonnegative solutions $(x,y)$ of $x^2-pqy^2=1$ are given by the set
\begin{align*}
\{(1,0), (x_n,y_n) : x_{n+1}=x_1x_n+pqy_1y_n, y_{n+1}=x_1y_n+y_1x_n, n=1,2,\ldots\}.
\end{align*}
Again, by mathematical induction, we have
\begin{align*}
\begin{bmatrix}
x_1' & qy_1' \\
py_1' & x_1'
\end{bmatrix}^{2n-1}&=\begin{bmatrix}
x_{n}' & qy_{n}' \\
py_{n}' & x_{n}'
\end{bmatrix},\\
\begin{bmatrix}
x_1' & qy_1' \\
py_1' & x_1'
\end{bmatrix}^{-(2n-1)}&=\begin{bmatrix}
-x_{n}' & qy_{n}' \\
py_{n}' & -x_{n}'
\end{bmatrix}, n \in \bbN.
\end{align*}
We also have,
\begin{align*}
\begin{bmatrix}
x_1' & qy_1' \\
py_1' & x_1'
\end{bmatrix}^{2n}&=\begin{bmatrix}
x_1'^2+pqy_1'^2 & 2q2x_1'y_1' \\
2px_1'y_1' & x_1'^2+pqy_1'^2
\end{bmatrix}^n=\begin{bmatrix}
x_1 & qy_1 \\
py_1 & x_1
\end{bmatrix}^n=\begin{bmatrix}
x_{n} & qy_{n} \\
py_{n} & x_{n}
\end{bmatrix}
\end{align*}
and
\begin{align*}
\begin{bmatrix}
x_1' & qy_1' \\
py_1' & x_1'
\end{bmatrix}^{-2n}&=\begin{bmatrix}
x_{n} & q(-y_{n}) \\
p(-y_{n}) & x_{n}
\end{bmatrix}, n \in \bbN.
\end{align*}
If $a=1$ and $j=0$, then $A=\begin{bmatrix}
1 & 0 \\
0 & 1
\end{bmatrix}$.
Thus  $A=\begin{bmatrix}
x_1' & qy_1' \\
py_1' & x_1'
\end{bmatrix}^n$, for some $n \in \bbZ$.
\end{proof}

\begin{proposition}\label{p:sodd}
Let $\alpha$ be a positive irrational number of the form $\frac{r}{s}+k\sqrt{\frac{p}{q}}$, where $r \in \bbZ, p,q,s \in \bbN$, $k \in \{-1,1\}$, $s$ is odd, $\gcd(r,s)=1$ and $\gcd(p,q)=1$. Let
$d_1=\gcd(ps^2-qr^2,qs)$.\\
(1) If the equation $x^2-\frac{pqs^4}{d_1^2}y^2=-1$ is not solvable over the integers, then, for any $A \in GL(2,\bbZ)$,
$\pi(A)$ is an automorphism of $\sA_{\al}$ if and only if $A=\begin{bmatrix}
-\frac{qrs}{d_1}y_1+x_1 & \frac{qs^2}{d_1}y_1 \\
(\frac{ps^2-qr^2}{d_1})y_1 & \frac{qrs}{d_1}y_1+x_1
\end{bmatrix}^n$, for some $n \in \bbZ$, where $(x_1,y_1)$ is the fundamental solution of $x^2-\frac{pqs^4}{d_1^2}y^2=1$.\\
(2) If the equation  $x^2-\frac{pqs^4}{d_1^2}y^2=-1$ is solvable over the integers, then, for any $A \in GL(2,\bbZ)$,
$\pi(A)$ is an automorphism of $\sA_{\al}$ if and only if $A=\begin{bmatrix}
-\frac{qrs}{d_1}ky_1'+x_1' & \frac{qs^2}{d_1}ky_1' \\
(\frac{ps^2-qr^2}{d_1})ky_1' & \frac{qrs}{d_1}ky_1'+x_1'
\end{bmatrix}^n$, for some $n \in \bbZ$, where $(x_1',y_1')$ is the fundamental solution of $x^2-\frac{pqs^4}{d_1^2}y^2=-1$.
\end{proposition}

\begin{proof}
First, note that $\frac{qrs}{d_1},\frac{qs^2}{d_1},\frac{ps^2-qr^2}{d_1} \in \bbZ$.  We will show that $\frac{pqs^4}{d_1^2}$ is a nonsquare integer. Since $d_1|(ps^2-qr^2)$ and $d_1|qs^2$, we have $d_1^2|(pqs^4-q^2r^2s^2)$. We also have that $d_1^2|q^2r^2s^2$. Thus $d_1^2|pqs^4$. Since $\al$ is irrational, $\sqrt{\frac{p}{q}}$ is irrational. Thus $pq$ is nonsquare. Hence $\frac{pqs^4}{d_1^2}$ must be nonsquare.
Now, let $A=\begin{bmatrix}
a & b \\
c & d
\end{bmatrix} \in GL(2,\bbZ)$ be such that $a+b\alpha>0$ and
\begin{align*}
b\alpha^2+(a-d)\alpha-c=0.
\end{align*}
Then $a-d=-2\frac{r}{s}b$ and $c=(\frac{ps^2-qr^2}{qs^2})b$. So $A$ is of the form
\begin{align*}
\begin{bmatrix}
a & b \\
(\frac{ps^2-qr^2}{qs^2})b & a+2\frac{r}{s}b
\end{bmatrix}.
\end{align*}
Since $A \in GL(2,\bbZ)$ and $\gcd(2r,s)=1$, we have $s|b$. Then there exists $j \in \bbZ$ such that $b=sj$. Now,
\begin{align*}
A=\begin{bmatrix}
a & sj \\
(\frac{ps^2-qr^2}{qs})j & a+2rj
\end{bmatrix}.
\end{align*}
Since $\gcd(\frac{ps^2-qr^2}{d_1},\frac{qs}{d_1})=1$, $\frac{qs}{d_1}|j$. So $j=\frac{qs}{d_1}l$ for some $l\in \bbZ$. Thus
\begin{align*}
A=\begin{bmatrix}
a & \frac{qs^2}{d_1}l \\
(\frac{ps^2-qr^2}{d_1})l & a+2\frac{qrs}{d_1}l
\end{bmatrix}.
\end{align*}
If $\operatorname{det}(A)=1$, then $a^2+2\frac{qrs}{d_1}la-(\frac{ps^2-qr^2}{d_1})\frac{qs^2}{d_1}l^2=1$. By the quadratic formula,
\begin{align*}
a=-\frac{qrs}{d_1}l\pm \sqrt{\frac{pqs^4}{d_1^2}l^2+1}.
\end{align*}
If $a=-\frac{qrs}{d_1}l+\sqrt{\frac{pqs^4}{d_1^2}l^2+1}$, then
\begin{align*}
a+b(\frac{r}{s}+k\sqrt{\frac{p}{q}})&=-\frac{qrs}{d_1}l+\sqrt{\frac{pqs^4}{d_1^2}l^2+1}+ \frac{qs^2}{d_1}l(\frac{r}{s}+k\sqrt{\frac{p}{q}})\\
&=\sqrt{\frac{pqs^4}{d_1^2}l^2+1}+ k\frac{s^2\sqrt{pq}}{d_1}l>0.
\end{align*}
If $a=-\frac{qrs}{d_1}l-\sqrt{\frac{pqs^4}{d_1^2}l^2+1}$, then
\begin{align*}
a+b(\frac{r}{s}+k\sqrt{\frac{p}{q}})&=-\frac{qrs}{d_1}l-\sqrt{\frac{pqs^4}{d_1^2}l^2+1}+ \frac{qs^2}{d_1}l(\frac{r}{s}+k\sqrt{\frac{p}{q}})\\
&=-\sqrt{\frac{pqs^4}{d_1^2}l^2+1}+ k\frac{s^2\sqrt{pq}}{d_1}l<0.
\end{align*}
Hence $a=-\frac{qrs}{d_1}l+\sqrt{\frac{pqs^4}{d_1^2}l^2+1}$ . Now $A$ is of the form
\begin{align*}
\begin{bmatrix}
-\frac{qrs}{d_1}l+\sqrt{\frac{pqs^4}{d_1^2}l^2+1} & \frac{qs^2}{d_1}l \\
(\frac{ps^2-qr^2}{d_1})l & \frac{qrs}{d_1}l+\sqrt{\frac{pqs^4}{d_1^2}l^2+1}
\end{bmatrix}.
\end{align*}
Since $A \in GL(2, \bbZ)$, $\sqrt{\frac{pqs^4}{d_1^2}l^2+1}$ must be an integer. Thus
\begin{align*}
A=\begin{bmatrix}
-\frac{qrs}{d_1}l+x & \frac{qs^2}{d_1}l \\
(\frac{ps^2-qr^2}{d_1})l & \frac{qrs}{d_1}l+x
\end{bmatrix},
\end{align*}
where $x,l\in \bbZ, x>0$ and $x^2-\frac{pqs^4}{d_1^2}l^2=1$. In this case, we have to find integers $x ,l$ such that $x>0$ and $x^2-\frac{pqs^4}{d_1^2}l^2=1$.\\
If $\operatorname{det}(A)=-1$, then $a^2+2\frac{qrs}{d_1}la-(\frac{ps^2-qr^2}{d_1})\frac{qs^2}{d_1}l^2=-1$. So
\begin{align*}
a=-\frac{qrs}{d_1}l\pm \sqrt{\frac{pqs^4}{d_1^2}l^2-1}.
\end{align*}
If $a=-\frac{qrs}{d_1}l+\sqrt{\frac{pqs^4}{d_1^2}l^2-1}$, then
\begin{align*}
a+b(\frac{r}{s}+k\sqrt{\frac{p}{q}})&=-\frac{qrs}{d_1}l+\sqrt{\frac{pqs^4}{d_1^2}l^2-1}+ \frac{qs^2}{d_1}l(\frac{r}{s}+k\sqrt{\frac{p}{q}})\\
&=\sqrt{\frac{pqs^4}{d_1^2}l^2-1}+ k\frac{s^2\sqrt{pq}}{d_1}l \begin{cases} <0
&\mbox{if }k=1 \text{ and } l<0\\
>0 & \mbox{if }k=1 \text{ and } l>0\\
>0 & \mbox{if }k=-1 \text{ and } l<0\\
<0 & \mbox{if }k=-1 \text{ and } l>0
\end{cases}.
\end{align*}
If $a=-\frac{qrs}{d_1}l-\sqrt{\frac{pqs^4}{d_1^2}l^2-1}$, then
\begin{align*}
a+b(\frac{r}{s}+k\sqrt{\frac{p}{q}})&=-\frac{qrs}{d_1}l-\sqrt{\frac{pqs^4}{d_1^2}l^2-1}+ \frac{qs^2}{d_1}l(\frac{r}{s}+k\sqrt{\frac{p}{q}})\\
&=-\sqrt{\frac{pqs^4}{d_1^2}l^2-1}+ k\frac{s^2\sqrt{pq}}{d_1}l\begin{cases} <0
&\mbox{if }k=1 \text{ and } l<0\\
>0 & \mbox{if }k=1 \text{ and } l>0\\
>0 & \mbox{if }k=-1 \text{ and } l<0\\
<0 & \mbox{if }k=-1 \text{ and } l>0
\end{cases}.
\end{align*}
Since $A \in GL(2, \bbZ)$, $\sqrt{\frac{pqs^4}{d_1^2}l^2-1}$ must be an integer. If $k=1$, then
\begin{align*}
A=\begin{bmatrix}
-\frac{qrs}{d_1}l+x & \frac{qs^2}{d_1}l \\
(\frac{ps^2-qr^2}{d_1})l & \frac{qrs}{d_1}l+x
\end{bmatrix},
\end{align*}
where $x,l\in \bbZ, l>0$ and $x^2-\frac{pqs^4}{d_1^2}l^2=-1$.
In this case, we have to find integers $x ,l$ such that $l>0$ and $x^2-\frac{pqs^4}{d_1^2}l^2=-1$.
If $k=-1$, then
\begin{align*}
A=\begin{bmatrix}
-\frac{qrs}{d_1}l+x & \frac{qs^2}{d_1}l \\
(\frac{ps^2-qr^2}{d_1})l & \frac{qrs}{d_1}l+x
\end{bmatrix},
\end{align*}
where $x,l\in \bbZ, l<0$ and $x^2-\frac{pqs^4}{d_1^2}l^2=-1$.
In this case, we have to find integers $x ,l$ such that $l<0$ and $x^2-\frac{pqs^4}{d_1^2}l^2=-1$.\\
Case (1) :  The equation $x^2-\frac{pqs^4}{d_1^2}y^2=-1$ is not solvable over the integers.\\
Then $\operatorname{det}(A)$ must be 1. Thus $A=\begin{bmatrix}
-\frac{qrs}{d_1}l+x & \frac{qs^2}{d_1}l \\
(\frac{ps^2-qr^2}{d_1})l & \frac{qrs}{d_1}l+x
\end{bmatrix}$, for some $x,l\in \bbZ, x>0$ and $x^2-\frac{pqs^4}{d_1^2}l^2=1$.
Let $(x_1,y_1)$ be the fundamental solution of the Pell's equation $x^2-\frac{pqs^4}{d_1^2}y^2=1$. Then all nonnegative solutions $(x,y)$ are given by the set
\begin{align*}
\{(1,0), (x_n,y_n) : &x_{n+1}=x_1x_n+\frac{pqs^4}{d_1^2}y_1y_n,\\
&y_{n+1}=x_1y_n+y_1x_n, n=1,2,\ldots\}.
\end{align*}
Thus $A=\begin{bmatrix}
-\frac{qrs}{d_1}y_1+x_1 & \frac{qs^2}{d_1}y_1 \\
(\frac{ps^2-qr^2}{d_1})y_1 & \frac{qrs}{d_1}y_1+x_1
\end{bmatrix}^n$, for some $n \in \bbZ$.\\
Case (2) :  The equation $x^2-\frac{pqs^4}{d_1^2}y^2=-1$ is solvable over the integers.\\
If $k=1$, then $A$ must be of the form $\begin{bmatrix}
-\frac{qrs}{d_1}l+x & \frac{qs^2}{d_1}l \\
(\frac{ps^2-qr^2}{d_1})l & \frac{qrs}{d_1}l+x
\end{bmatrix}$ where $x,l \in \bbZ$,
\begin{align*}
x>0, \text{ and } x^2-\frac{pqs^4}{d_1^2}l^2=1
\end{align*}
or
\begin{align*}
l>0, \text{ and } x^2-\frac{pqs^4}{d_1^2}l^2=-1.
\end{align*}
If $k=-1$, then $A$ must be of the form $\begin{bmatrix}
-\frac{qrs}{d_1}l+x & \frac{qs^2}{d_1}l \\
(\frac{ps^2-qr^2}{d_1})l & \frac{qrs}{d_1}l+x
\end{bmatrix}$ where $x,l \in \bbZ$,
\begin{align*}
x>0, \text{ and } x^2-\frac{pqs^4}{d_1^2}l^2=1
\end{align*}
or
\begin{align*}
l<0, \text{ and } x^2-\frac{pqs^4}{d_1^2}l^2=-1.
\end{align*}
Let $(x_1',y_1')$ be the fundamental solution of the negative Pell's equation $x^2-\frac{pqs^4}{d_1^2}y^2=-1$. Then $(x_1,y_1)=(x_1'^2+\frac{pqs^4}{d_1^2}y_1'^2, 2x_1'y_1')$ is the fundamental solution of $x^2-\frac{pqs^4}{d_1^2}y^2=1$. Also, all positive solutions $(x,y)$ of $x^2-\frac{pqs^4}{d_1^2}y^2=-1$ are given by the set
\begin{align*}
\{(x_n',y_n') : &x_{n+1}'=(x_1'^2+\frac{pqs^4}{d_1^2}y_1'^2)x_n'+2\frac{pqs^4}{d_1^2}x_1'y_1'y_n',\\ &y_{n+1}'=(x_1'^2+\frac{pqs^4}{d_1^2}y_1'^2)y_n'+2x_1'y_1'x_n', n=1,2,\ldots\}
\end{align*}
and  all nonnegative solutions $(x,y)$ of $x^2-\frac{pqs^4}{d_1^2}y^2=1$ are given by the set
\begin{align*}
\{(1,0), (x_n,y_n) : &x_{n+1}=x_1x_n+\frac{pqs^4}{d_1^2}y_1y_n,\\
&y_{n+1}=x_1y_n+y_1x_n, n=1,2,\ldots\}.
\end{align*}
Thus  $A=\begin{bmatrix}
-\frac{qrs}{d_1}ky_1'+x_1' & \frac{qs^2}{d_1}ky_1' \\
(\frac{ps^2-qr^2}{d_1})ky_1' & \frac{qrs}{d_1}ky_1'+x_1'
\end{bmatrix}^n$, for some $n \in \bbZ$.
\end{proof}

\begin{proposition}\label{p:seven}
Let $\alpha$ be a positive irrational number of the form $\frac{r}{s}+k\sqrt{\frac{p}{q}}$, where $r \in \bbZ, p,q,s \in \bbN$, $k \in \{-1,1\}$, $s$ is even, $\gcd(r,s)=1$ and $\gcd(p,q)=1$. Let
$d_1=\gcd(ps^2-qr^2,2qs)$.\\
(1) If $d_1\,|\,qs$ and the equation $x^2-\frac{pqs^4}{d_1^2}y^2=-1$ is not solvable over the integers, then, for any $A \in GL(2,\bbZ)$,
$\pi(A)$ is an automorphism of $\sA_{\al}$ if and only if $A=\begin{bmatrix}
-\frac{qrs}{d_1}y_1+x_1 & \frac{qs^2}{d_1}y_1 \\
(\frac{ps^2-qr^2}{d_1})y_1 & \frac{qrs}{d_1}y_1+x_1
\end{bmatrix}^n$, for some $n \in \bbZ$, where $(x_1,y_1)$ is the fundamental solution of $x^2-\frac{pqs^4}{d_1^2}y^2=1$.\\
(2) If $d_1\,|\,qs$ and the equation  $x^2-\frac{pqs^4}{d_1^2}y^2=-1$ is solvable over the integers, then, for any $A \in GL(2,\bbZ)$,
$\pi(A)$ is an automorphism of $\sA_{\al}$ if and only if $A=\begin{bmatrix}
-\frac{qrs}{d_1}ky_1'+x_1' & \frac{qs^2}{d_1}ky_1' \\
(\frac{ps^2-qr^2}{d_1})ky_1' & \frac{qrs}{d_1}ky_1'+x_1'
\end{bmatrix}^n$, for some $n \in \bbZ$, where $(x_1',y_1')$ is the fundamental solution of $x^2-\frac{pqs^4}{d_1^2}y^2=-1$.\\
(3) If $d_1 \nmid qs$ and the equation $x^2-\frac{4pqs^4}{d_1^2}y^2=-4$ is not solvable over the integers, then, for any $A \in GL(2,\bbZ)$,
$\pi(A)$ is an automorphism of $\sA_{\al}$ if and only if $A=\begin{bmatrix}
-\frac{qrs}{d_1}y_1+\frac{x_1}{2} & \frac{qs^2}{d_1}y_1\\
(\frac{ps^2-qr^2}{d_1})y_1 & \frac{qrs}{d_1}y_1+\frac{x_1}{2}
\end{bmatrix}^n$, for some $n \in \bbZ$, where $(x_1,y_1)$ is the fundamental solution of $x^2-\frac{4pqs^4}{d_1^2}y^2=4$.\\
(4) If $d_1 \nmid qs$ and the equation $x^2-\frac{4pqs^4}{d_1^2}y^2=-4$ is solvable over the integers, then, for any $A \in GL(2,\bbZ)$,
$\pi(A)$ is an automorphism of $\sA_{\al}$ if and only if  $A=\begin{bmatrix}
-\frac{qrs}{d_1}ky_1'+\frac{x_1'}{2} & \frac{qs^2}{d_1}ky_1'\\
(\frac{ps^2-qr^2}{d_1})ky_1' & \frac{qrs}{d_1}ky_1'+\frac{x_1'}{2}
\end{bmatrix}^n$, for some $n \in \bbZ$, where $(x_1',y_1')$ is the fundamental solution of $x^2-\frac{4pqs^4}{d_1^2}y^2=-4$.
\end{proposition}

\begin{proof}
First, note that if $d_1 \,|\,qs$, then $\frac{qrs}{d_1},\frac{qs^2}{d_1},\frac{ps^2-qr^2}{d_1} \in \bbZ$. In this case, we will show that $\frac{pqs^4}{d_1^2}$ is a nonsquare integer. Since $d_1|(ps^2-qr^2)$ and $d_1|qs^2$, we have $d_1^2|(pqs^4-q^2r^2s^2)$. We also have that $d_1^2|q^2r^2s^2$. Thus $d_1^2|pqs^4$. Since $\al$ is irrational, $\sqrt{\frac{p}{q}}$ is irrational. Thus $pq$ is nonsquare. Hence $\frac{pqs^4}{d_1^2}$ must be nonsquare. Now, if $d_1 \nmid qs$, then $\frac{ps^2-qr^2}{d_1},\frac{qs^2}{d_1} \in \bbZ$. In this case, we will show that$\frac{4pqs^4}{d_1^2}$ is a nonsquare integer. Since $d_1|2qs$, $d_1|4qs^2$. Since $d_1|(ps^2-qr^2)$ and $d_1|4qs^2$, we have $d_1^2|(4pqs^4-4q^2r^2s^2)$. We also have that $d_1^2|4q^2r^2s^2$. Thus $d_1^2|4pqs^4$. Since $\al$ is irrational, $\sqrt{\frac{p}{q}}$ is irrational. Thus $pq$ is nonsquare. Hence $\frac{4pqs^4}{d_1^2}$ must be nonsquare.

Now, let $A=\begin{bmatrix}
a & b \\
c & d
\end{bmatrix} \in GL(2,\bbZ)$ be such that $a+b\alpha>0$ and
\begin{align*}
b\alpha^2+(a-d)\alpha-c=0.
\end{align*}
Then $a-d=-2\frac{r}{s}b$ and $c=(\frac{ps^2-qr^2}{qs^2})b$. So $A$ is of the form
\begin{align*}
\begin{bmatrix}
a & b \\
(\frac{ps^2-qr^2}{qs^2})b & a+2\frac{r}{s}b
\end{bmatrix}.
\end{align*}
Since $s$ is even, $\frac{s}{2} \in \bbN$. Since $A \in GL(2,\bbZ)$ and $\gcd(r,\frac{s}{2})=1$, we have $\frac{s}{2}|b$. Then there exists $j \in \bbZ$ such that $b=\frac{s}{2}j$. Now,
\begin{align*}
A=\begin{bmatrix}
a & \frac{s}{2}j \\
(\frac{ps^2-qr^2}{2qs})j & a+rj
\end{bmatrix}.
\end{align*}
Since $\gcd(\frac{ps^2-qr^2}{d_1},\frac{2qs}{d_1})=1$, $\frac{2qs}{d_1}|j$. So $j=\frac{2qs}{d_1}l$ for some $l\in \bbZ$. Thus
\begin{align*}
A=\begin{bmatrix}
a & \frac{qs^2}{d_1}l \\
(\frac{ps^2-qr^2}{d_1})l & a+2\frac{qrs}{d_1}l
\end{bmatrix}.
\end{align*}
Case : $d_1|qs$.\\
By the same argument in the proof of Proposition~\ref{p:sodd},
\begin{align*}
A=\begin{bmatrix}
-\frac{qrs}{d_1}y_1+x_1 & \frac{qs^2}{d_1}y_1 \\
(\frac{ps^2-qr^2}{d_1})y_1 & \frac{qrs}{d_1}y_1+x_1
\end{bmatrix}^n,
\end{align*}
for some $n \in \bbZ$, where $(x_1,y_1)$ is the fundamental solution of $x^2-\frac{pqs^4}{d_1^2}y^2=1$ if the equation $x^2-\frac{pqs^4}{d_1^2}y^2=-1$ is not solvable over the integers and
\begin{align*}
A=\begin{bmatrix}
-\frac{qrs}{d_1}ky_1'+x_1' & \frac{qs^2}{d_1}ky_1' \\
(\frac{ps^2-qr^2}{d_1})ky_1' & \frac{qrs}{d_1}ky_1'+x_1'
\end{bmatrix}^n,
\end{align*}
for some $n \in \bbZ$, where $(x_1',y_1')$ is the fundamental solution of $x^2-\frac{pqs^4}{d_1^2}y^2=-1$ if the equation $x^2-\frac{pqs^4}{d_1^2}y^2=-1$ is solvable over the integers.\\
Case : $d_1 \nmid qs$.\\
Recall that $A=\begin{bmatrix}
a & \frac{qs^2}{d_1}l \\
(\frac{ps^2-qr^2}{d_1})l & a+2\frac{qrs}{d_1}l
\end{bmatrix}$, for some $l \in \bbZ$\\
If $\operatorname{det}(A)=1$, then $a^2+2\frac{qrs}{d_1}la-(\frac{ps^2-qr^2}{d_1})\frac{qs^2}{d_1}l^2=1$. By the quadratic formula,
\begin{align*}
a=-\frac{qrs}{d_1}l\pm\frac{ \sqrt{\frac{4pqs^4}{d_1^2}l^2+4}}{2}.
\end{align*}
Since $a+b\al>0$, $a=-\frac{qrs}{d_1}l+\frac{ \sqrt{\frac{4pqs^4}{d_1^2}l^2+4}}{2}$ . Now $A$ is of the form
\begin{align*}
\begin{bmatrix}
-\frac{qrs}{d_1}l+\frac{ \sqrt{\frac{4pqs^4}{d_1^2}l^2+4}}{2} & \frac{qs^2}{d_1}l \\
(\frac{ps^2-qr^2}{d_1})l & \frac{qrs}{d_1}l+\frac{ \sqrt{\frac{4pqs^4}{d_1^2}l^2+4}}{2}
\end{bmatrix}.
\end{align*}
\indent Since $d_1 \nmid qs$ and $d_1\,|\,2qs$, $d_1$ must be even. Since $\gcd(r,s)=1$ and $s$ is even, $r$ must be odd. Since $d_1=\gcd(ps^2-qr^2,2qs)$ and $d_1$ is even, $q$ must be even. Since $\gcd(p,q)=1$ and $q$ is even, $p$ must be odd. Now we write $s=2^{m_1}n_1,q=2^{m_2}n_2$, where $m_1,m_2 \in \bbN$ and $n_1,n_2$ are odd integers. Since $d_1 \nmid qs$ and $d_1\,|\,2qs$, $d_1=2^{m_1+m_2+1}n_3$ for some odd integer $n_3$ and $n_3\,|\,n_1n_2$. Since $n_1,n_2,n_3$ are odd and $n_3\,|\,n_1n_2$, $\frac{n_1n_2}{n_3}$ is an odd integer. Now we have
\begin{align*}
\frac{qrs}{d_1}=\frac{2^{m_2}n_2r2^{m_1}n_1}{2^{m_1+m_2+1}n_3}=\frac{r(\frac{n_1n_2}{n_3})}{2}.
\end{align*}
Thus $A$ is of the form
\begin{align*}
\begin{bmatrix}
\frac{-r(\frac{n_1n_2}{n_3})l+\sqrt{\frac{4pqs^4}{d_1^2}l^2+4}}{2} & \frac{qs^2}{d_1}l \\
(\frac{ps^2-qr^2}{d_1})l & \frac{r(\frac{n_1n_2}{n_3})l+\sqrt{\frac{4pqs^4}{d_1^2}l^2+4}}{2}
\end{bmatrix}.
\end{align*}
Since $A \in GL(2, \bbZ)$,  $\sqrt{\frac{4pqs^4}{d_1^2}l^2+4}$ must be an integer. Thus
\begin{align*}
A=\begin{bmatrix}
\frac{-r(\frac{n_1n_2}{n_3})l+x}{2} & \frac{qs^2}{d_1}l \\
(\frac{ps^2-qr^2}{d_1})l & \frac{r(\frac{n_1n_2}{n_3})l+x}{2}
\end{bmatrix}.
\end{align*}
where $x,l\in \bbZ, x>0$ and $x^2-\frac{4pqs^4}{d_1^2}l^2=4$.\\
\indent Since $\frac{-r(\frac{n_1n_2}{n_3})l+x}{2}$ and $\frac{r(\frac{n_1n_2}{n_3})l+x}{2}$ must be integers and $r,\frac{n_1n_2}{n_3}$ are odd, $x$ and $l$ must have the same parity. Note that if $n \not\equiv 0\ (\textrm{mod}\ 4)$, then any solution to the equation $x^2-ny^2=4$ or $x^2-ny^2=-4$, $x$ and $y$ have the same parity. Now, we will show that $\frac{4pqs^4}{d_1^2} \not\equiv 0\ (\textrm{mod}\ 4)$, i.e., we have to show that $d_1^2 \nmid pqs^4$. Since $d_1\,|\,ps^2-qr^2$ and $d_1=2^{m_1+m_2+1}n_3$, we have $2^{m_1+m_2+1}\,|\,ps^2-qr^2$. Let $m=\min\{2m_1,m_2\}$. Then
\begin{align*}
ps^2-qr^2=2^m(p\frac{s^2}{2^m}-\frac{q}{2^m}r).
\end{align*}
Since  $m=\min\{2m_1,m_2\}$, $p\frac{s^2}{2^m}$ or $\frac{q}{2^m}r$ must be odd. Since $m \leq m_2 <m_1+m_2+1$, $p\frac{s^2}{2^m}-\frac{q}{2^m}r$ must be even. That is $p\frac{s^2}{2^m}$ and $\frac{q}{2^m}r$ must be odd. Thus $m=m_2=2m_1$. Now we have
\begin{align*}
pqs^4=p2^{m_2}n_22^{4m_1}n_1^4=2^{m_2+4m_1}pn_2n_1^4
\end{align*}
and
\begin{align*}
d_1^2=2^{2m_1+2m_2+2}n_3^2.
\end{align*}
Since $m_2+4m_1=2m_2+2m_1+1<2m_1+2m_2+2$, $d_1^2 \nmid pqs^4$. Thus $\frac{4pqs^4}{d_1^2} \not\equiv 0\ (\textrm{mod}\ 4)$. $x$ and $l$ that satisfy the equation $x^2-\frac{4pqs^4}{d_1^2}l^2=4$ have the same parity. Thus, in this case, we have to find integers $x ,l$ such that $x>0$ and $x^2-\frac{4pqs^4}{d_1^2}l^2=4$.\\
If $\operatorname{det}(A)=-1$, by the same argument above, if $k=1$, then
\begin{align*}
A=\begin{bmatrix}
-\frac{qrs}{d_1}l+\frac{x}{2} & \frac{qs^2}{d_1}l\\
(\frac{ps^2-qr^2}{d_1})l & \frac{qrs}{d_1}l+\frac{x}{2}
\end{bmatrix},
\end{align*}
where $x,l\in \bbZ, l>0$ and $x^2-\frac{4pqs^4}{d_1^2}l^2=-4$.\\
If $k=-1$, then
\begin{align*}
A=\begin{bmatrix}
-\frac{qrs}{d_1}l+\frac{x}{2} & \frac{qs^2}{d_1}l\\
(\frac{ps^2-qr^2}{d_1})l & \frac{qrs}{d_1}l+\frac{x}{2}
\end{bmatrix},
\end{align*}
where $x,l\in \bbZ, l<0$ and $x^2-\frac{4pqs^4}{d_1^2}l^2=-4$.\\
In this case, we have to find integers $x ,l$ such that $kl>0$ and $x^2-\frac{4pqs^4}{d_1^2}l^2=-4$.\\
 By the same argument in the proof of Proposition~\ref{p:sodd}, we have
\begin{align*}
A=\begin{bmatrix}
-\frac{qrs}{d_1}y_1+\frac{x_1}{2} & \frac{qs^2}{d_1}y_1\\
(\frac{ps^2-qr^2}{d_1})y_1 & \frac{qrs}{d_1}y_1+\frac{x_1}{2}
\end{bmatrix}^n,
\end{align*}
for some $n \in \bbZ$, where $(x_1,y_1)$ is the fundamental solution of $x^2-\frac{4pqs^4}{d_1^2}l^2=4$ if the equation $x^2-\frac{4pqs^4}{d_1^2}y^2=-4$ is not solvable over the integers and
\begin{align*}
A=\begin{bmatrix}
-\frac{qrs}{d_1}ky_1'+\frac{x_1'}{2} & \frac{qs^2}{d_1}ky_1'\\
(\frac{ps^2-qr^2}{d_1})ky_1' & \frac{qrs}{d_1}ky_1'+\frac{x_1'}{2}
\end{bmatrix}^n,
\end{align*}
for some $n \in \bbZ$, where $(x_1',y_1')$ is the fundamental solution of $x^2-\frac{4pqs^4}{d_1^2}y^2=-4$ if the equation $x^2-\frac{4pqs^4}{d_1^2}y^2=-4$ is solvable over the integers.
\end{proof}

\begin{example}
$\alpha=\sqrt{5}$.\\
\indent The fundamental solution of the equation $x^2-5y^2=-1$ is $(x_1',y_1')=(2,1)$. Thus any matrix $A$ in $GL(2,\bbZ)$ that $\pi(A)$ is automorphisms of $\sA_{\sqrt{5}}$ is of the form $A=\begin{bmatrix}
2 & 1 \\
5 & 2
\end{bmatrix}^n$, $n \in \bbZ$, i.e., $\pi(A)$ is an automorphism of $A_{\sqrt{5}}$ if and only if $A=\begin{bmatrix}
2 & 1 \\
5 & 2
\end{bmatrix}^n$ for some $n \in \bbZ$.

\end{example}

\begin{example}
$\alpha=\sqrt{7}$.\\
\indent Since $7 \not\equiv 1 (\mod 4)$, the equation $x^2-7y^2=-1$ is not solvable over integers. So we look at the fundamental solution of the equation  $x^2-7y^2=1$. Since $(x_1,y_1)=(8,3)$ is the fundamental solution of $x^2-7y^2=1$,  any matrix $A$ in $GL(2,\bbZ)$ that $\pi(A)$ is automorphisms of $\sA_{\sqrt{7}}$ is of the form $A=\begin{bmatrix}
8 & 3 \\
21 & 8
\end{bmatrix}^n$, $n \in \bbZ$.
\end{example}

\begin{example}
$\alpha=\frac{1+\sqrt{7}}{3}$.\\
\indent Since $(x_1,y_1)=(8,3)$ is the fundamental solution of $x^2-7y^2=1$, by Proposition~\ref{p:sodd}, any matrix $A$ in $GL(2,\bbZ)$ that $\pi(A)$ is automorphisms of $\sA_{\frac{1+\sqrt{7}}{3}}$ is of the form $A=\begin{bmatrix}
5 & 9 \\
6 & 11
\end{bmatrix}^n$, $n \in \bbZ$.
\end{example}

\begin{example}
$\alpha=\frac{1+\sqrt{5}}{2}$.\\
\indent The fundamental solution of the equation $x^2-5y^2=-4$ is $(x_1',y_1')=(1,1)$. By Proposition~\ref{p:seven}, any matrix $A$ in $GL(2,\bbZ)$ that $\pi(A)$ is automorphisms of $\sA_{\frac{1+\sqrt{5}}{2}}$ is of the form $A=\begin{bmatrix}
0 & 1 \\
1 & 1
\end{bmatrix}^n$, $n \in \bbZ$
\end{example}

Note that $\begin{bmatrix}
2 & 1 \\
5 & 2
\end{bmatrix}$ is not conjugate to $\begin{bmatrix}
0 & 1 \\
1 & 1
\end{bmatrix}$ nor $\begin{bmatrix}
-1 & 1 \\
1 & 0
\end{bmatrix}=\begin{bmatrix}
0 & 1 \\
1 & 1
\end{bmatrix}^{-1}$ since they don't have the same trace. So $\sA_{\sqrt{5}}$ is not isomorphic to $\sA_{\frac{1+\sqrt{5}}{2}}$.
\bibliographystyle{plain}

\end{document}